\documentclass[12pt]{article}

\usepackage[english]{babel}

\usepackage[letterpaper,top=2cm,bottom=2cm,left=3cm,right=3cm,marginparwidth=1.75cm]{geometry}

\usepackage{amsmath}
\usepackage{graphicx}
\usepackage[colorlinks=true, linkcolor=black, allcolors=blue]{hyperref}
\usepackage{amssymb,amsthm}
\usepackage{amsfonts,graphicx,psfrag}
\usepackage{graphicx}
\usepackage{float}
\usepackage{color}
\usepackage{subcaption}

\newcommand{\fe}{\mathfrak{e}}

\newcommand{\rr}{\mathbb{R}}

\newtheorem{thm}{Theorem}[section]
\newtheorem{defi}{Definition}[section]
\newtheorem{prop}{Proposition}[section]
\newtheorem{lem}{Lemma}[section]
\newtheorem{cor}{Corollary}[section]
\newtheorem{rem}{Remark}[section]

\def\bea{\begin{eqnarray}}
\def\eea{\end{eqnarray}}
\def\no{\noindent}
\def\bs{\bigskip}
\def\Sp{{\mathrm {Sp}}}
\def\ga{{\gamma}}

\begin{document}
\title{Relative Periodic Orbits in the Gutzwiller-type Anisotropic Kepler Problem and $n$-body Problem}

\author{Xijun Hu\textsuperscript{1}
\quad Yuwei Ou\textsuperscript{1}
\quad Zhiwen Qiao\textsuperscript{1}
\quad Yifan Yang\textsuperscript{1}
\\ \\
\textsuperscript{1} School of Mathematics, Shandong University, 27 Shanda Nanlu,\\
250100 Jinan, P. R. China
\thanks{E-mails: xjhu@sdu.edu.cn, ywou@sdu.edu.cn, qiaozw@mail.sdu.edu.cn, yangyf@mail.sdu.edu.cn.}
}

\date{}

\maketitle

\begin{abstract}
We study the relative periodic orbits in the Gutzwiller-type anisotropic Kepler problem, which is a generalized model derived from the classical Gutzwiller anisotropic Kepler problem. By reducing the rotational symmetry of the $z$-axis, we obtain a reduced system with two degrees of freedom and its energy surface forms a compact and regular three-sphere in a certain parameter range. Combining the estimation of the Seifert rotation number of the planar Kepler orbit and the CHHL formula introduced in \cite{CHHL23}, we prove that this system admits infinitely many periodic orbits on every compact regular energy surface. This model can be applied to the $(1+2n)$-body problem to find infinitely many relative periodic orbits with hip-hop symmetry as well as relative periodic orbits in the $n$-pyramidal problem.
\end{abstract}

\bs

\no{\bf AMS Subject Classification:} 70F10, 37J46, 53D12, 53D10

\bs

\no{\bf Key Words: $n$-body problem, periodic orbits, hip-hop symmetry, rotation number, Gutzwiller anisotropic }

\tableofcontents
\section{Introduction and main results}
For the equal-mass $2n$-body problem with $n\geq2$, there exist solutions with hip-hop symmetry. With the center of mass of the $2n$-body system at the origin, such solutions satisfy the following condition: the $2n$ bodies are always partitioned into two groups of $n$, each occupying the vertices of a regular $n$-gon (or the endpoints of a line segment when $n=2$). These configurations lie in planes perpendicular to the $z$-axis, and their orthogonal projections onto the $(x,y)$-plane together form a regular $2n$-gon. These solutions were first numerically observed by Davies, Truman, and Williams \cite{davies1983classical}. Later, Chenciner and Venturelli \cite{Chenciner2000} used a variational approach to obtain collisionless $\mathbb{Z}/4\mathbb{Z}$-symmetric action minimizers and coined the term \emph{hip-hop} to describe orbits whose configuration oscillates between the central configuration of the square and that of the tetrahedron. This framework was subsequently extended to general $2n$-body systems. Specifically, Terracini and Venturelli \cite{Terracini2007} proved the existence of collisionless, non-planar and non-homographic trajectories with hip-hop symmetry for weak-force potentials $r^{-\sigma}$ by variational approach with topological constraints. Using an analytic continuation argument, Barrab\'{e}s, Cors, Pinyol, and Soler \cite{Barrabes2006,Barrabes2010} proved the existence of families of solutions with hip-hop symmetry that approach planar relative equilibria and highly eccentric elliptic relative equilibria in the Newtonian setting. 

This class of solutions can be generalized to the $(1+2n)$-body problem, where a central mass is introduced in addition to the $2n$ equal-mass bodies. There exists an invariant subsystem in which the $2n$ equal-mass bodies exhibit hip-hop symmetric motion, while the central body remains stationary at the system’s center of mass.

The Hamiltonian of the hip-hop $(1+2n)$-body problem in reduced form is given by
\begin{equation}\label{equ: hiphop ham}
	H_1(p_r,p_z,r,z)=\frac{1}{2}(p_r^{2}+p_z^{2})+V_1(r,z),
\end{equation}
with the corresponding potential
\begin{equation}
	\begin{aligned}
		V_1(r,z)&=\dfrac{\varpi^{2}}{2r^{2}}-\frac{1}{2}\sum_{k=1}^{2n-1}\dfrac{1}{\sqrt{4r^{2}\sin^{2} (\frac{k\pi}{2n})+((-1)^k -1)^2 z^{^{2}}}}-\dfrac{m_0}{\sqrt{r^{2}+z^{2}}}\\
		&=\frac{\varpi^{2}}{2r^{2}}-\frac{a(n)}{r}-\frac{1}{4}\sum_{k=1}^{n}\frac{b_{n,k}}{\sqrt{r^2+b_{n,k}^2z^2}}-\frac{m_0}{\sqrt{r^{2}+z^{2}}}.
	\end{aligned}
\end{equation}
Here, $(p_r,p_z,r,z)\in \mathbb{R}^{2}\times\mathbb{R}^{+}\times\mathbb{R}$ and $a(n)=\frac{1}{4}\sum_{k=1}^{n-1}\csc(\frac{k\pi}{n})$, $b_{n,k}=\csc(\frac{(2k-1)\pi}{2n})$. The parameter $\varpi$ denotes the angular momentum about $z$-axis, while $m_0$ stands for the mass ratio of the central mass to one of the other $2n$ masses. See Section \ref{sec: $1+2n$} for more details.

It should be specially pointed out that the periodic orbits of system \eqref{equ: hiphop ham} correspond to a family of relative periodic orbits in the system before reduction of angular momentum.

This Hamiltonian is a special case of the following reduced generalized Gutzwiller anisotropic Kepler problem,
\begin{equation}\label{equ: reduced Ham}
	H(p_r,p_z,r,z)=\frac{1}{2}(p_r^{2}+p_z^{2})+\frac{\varpi^2}{2r^2}-\frac{A_0}{r}-\sum_{i=1}^{\mathfrak{n}}\frac{A_i}{\sqrt{r^2+B_iz^2}},
\end{equation}
where $A_0\geq0$ and $A_i,B_i>0,\,i=1,\cdots,\mathfrak{n}$. We call this system the Gutzwiller-type anisotropic Kepler problem. Let $\omega=dp_r\wedge dr+dp_z\wedge dz$ be the standard symplectic form. The Hamiltonian vector field $X_{H}$ is determined by $dH=\omega(\cdot,X_{H})$ and the equation $\dot \zeta=X_{H}(\zeta)$ can be written as
\begin{equation}\label{equ: reduced Ham sys}
	\left\{\begin{aligned}
		&\dot{r}=p_r ,\quad  \dot{p_r}=\frac{\varpi^{2}}{r^{3}}-\frac{A_0}{r^2}-\sum_{i=1}^{\mathfrak{n}}\frac{A_ir}{(r^2+B_iz^2)^{3/2}},\\
		&\dot z =p_z ,\quad \dot{p_z}=-\sum_{i=1}^{\mathfrak{n}}\frac{A_iB_iz}{(r^2+B_iz^2)^{3/2}}.
	\end{aligned}
	\right.
\end{equation}

When $A_0=0$ and $\mathfrak{n}=1$, this Hamiltonian corresponds to the classical reduced Gutzwiller anisotropic Kepler problem,
\begin{equation}\label{aniso Kepler ham}
	H(p_r,p_z,r,z)=\frac{1}{2}(p_r^{2}+p_z^{2})+\frac{\varpi^2}{2r^2}-\frac{A_1}{\sqrt{r^2+B_1z^2}},
\end{equation}
where $B_1>0$ is the anisotropic parameter. This model was first studied by Gutzwiller \cite{GMC73,GMC1990}, and describes the motion of an electron in semiconductors with donor impurities in $\mathbb R^3$. When $B_1=1$, it corresponds to the reduced spatial Kepler problem. Guirao, Llibre and Vera \cite{GLV13} and Llibre and Makhlouf \cite{LM12} studied the periodic orbits by treating this system as a perturbed Kepler problem. Recently, Sakaguchi and Shibayama \cite{SakaShib26}, using a minimax approach, obtained a non-trivial symmetric periodic orbit on each compact energy surface for both the system \eqref{aniso Kepler ham} with $B_1>1$ and the isosceles three-body problem. Many researchers also considered the general case of the anisotropic Kepler problem containing the Gutzwiller-type anisotropic Kepler problem. Barutello, Terracini and Verzini \cite{BTG14} studied the exitence of the entire parabolic trajectories as a minimizer of the Lagrangian action functional for anisotropic Kepler problems in $\mathbb{R}^d$ with the $(-\alpha)$-homogeneous potentials. In \cite{HY18}, Hu and Yu developed an index theory for zero-energy solutions of the planar $(-\alpha)$-homogeneous anisotropic Kepler problem and established the relations between the Morse indices of zero-energy solutions and their oscillatory behaviors. Yu \cite{Y25} established the asymptotic properties for the positive-energy solutions of the anisotropic Kepler problem in $\mathbb R^d$ with a $(-\alpha)$-homogeneous potential, and proved certain existence results of hyperbolic and bi-hyperbolic solutions.  

In \cite{HOQ26}, the first three authors proved that the system \eqref{aniso Kepler ham} possesses infinitely many periodic orbits on any fixed compact regular energy surface for $B_1\in(0,1]$. In this paper, we prove that the system \eqref{equ: reduced Ham sys} admits infinitely many periodic orbits on its compact regular energy surface under a technical condition. As a special case, for the system \eqref{aniso Kepler ham}, the above conclusions remain valid when $B_1\in[1,+\infty)$. This generalizes the result in \cite{HOQ26}. Furthermore, this result can also be applied to the hip-hop $(1+2n)$-body problem and the $n$-pyramidal problem. Specifically, Proposition \ref{topo of energy surfaces 1} implies that $2h\varpi^2$ is an essential parameter of \eqref{equ: reduced Ham}. When $-C^2<2h\varpi^2<-A_0^2$ with $C=\sum_{i=0}^{\mathfrak{n}}A_i$, the energy surface $\mathfrak{M}=H^{-1}(h)$ forms a compact regular three-sphere. Since $H$ is a mechanical Hamiltonian, $\mathfrak M$ is a contact-type hypersurface in $\mathbb R^4$ \cite[Theorem 4.8]{HZ94}. There exists a contact $1$-form $\lambda$ satisfying $d\lambda=\omega$, and $\lambda \wedge d\lambda$ serves as a volume form on $\mathfrak M$. The Reeb vector field $R$, defined by $\lambda(R)=1$ and $d\lambda(R,\cdot)\equiv0$, preserves the contact structure $\xi:=\ker \lambda$. Moreover, the corresponding Reeb flow shares identical dynamics with $X_{H}$ up to reparametrization. There exists a special orbit $\zeta_p\subset\mathfrak{M}\cap\{p_z=z=0\}$, namely the planar Kepler orbit, which can be solved explicitly:
$$r_p(t)=\frac{\varpi^2/C}{1+\fe\cos\theta_{p}(t)},$$
where $\fe=(1+2h\varpi^2/C^2)^{1/2}$ denotes the eccentricity. The angle variable $\theta_{p}(t)$ obeys $r_p^2(t)\dot{\theta_{p}}=\varpi$ with the initial condition $\theta_{p}(0)=0$. Taking $\zeta_p$ as a Reeb orbit of the Reeb vector field $R$, we can derive the relation between the contact volume $\mathrm{vol}(\mathfrak M,\lambda):=\int_{\mathfrak M} \lambda \wedge d\lambda$ and the minimal Reeb period $T_p:=\int_{\zeta_p}\lambda=\int_{\Upsilon}dp_r\wedge dr$ of $\zeta_p$. Combined with the Seifert rotation number of $\zeta_p$, we can use the CHHL formula introduced in \cite{CHHL23} (see Section \ref{sec: Reeb flow and Rotation number}) to verify that infinitely many periodic orbits exist on the energy surface $\mathfrak M$.

The Seifert rotation number of a periodic orbit $\zeta\subset \mathfrak M$ is defined as $\hat \rho(\zeta)=\hat i(\zeta)/2\in \mathbb R$, where $\hat i(\zeta)$ is the mean index of the transverse flow determined by the Seifert surface spanned by $\zeta$. For an unknotted periodic orbit $\zeta$, $\hat \rho(\zeta)+1$ is equal to the usual rotation number $\rho(\zeta)$ considered in \cite{HLOYS23}. In particular, the rotation number $\hat \rho_p:=\hat \rho(\zeta_p)$ of the planar Kepler orbit $\zeta_p\subset \mathfrak M\cap \{p_z=z=0\}$ is determined by the following equation
\begin{equation}\label{equ: Hill stability equation}
	\ddot{x}+x+\beta(1+\fe\cos\theta)^{-1}x=0,
\end{equation}
named the Hill stability equation. The equation \eqref{equ: Hill stability equation} plays an important role in the stability of the homographic solutions in the $n$-body problem, see \cite{HOT23}. Moreover, the Hill stability equation \eqref{equ: Hill stability equation} is a special Ince's equation, see (7.3) and (7.27) in \cite{MW1966}.

Based on the volume computation in Section \ref{sec: the volume estimate} and the rotation number estimates in Section \ref{sec: the degenerate curves } and CHHL formula, we derive the following result.
\begin{thm}\label{thm: estimate of rotation number}
	Assuming $A_0\geq0,A_i,B_i>0,\,i=1,\cdots,\mathfrak{n}$ and $-C^2<2h\varpi^2<-A_0^2$ with $C=\sum_{i=0}^{\mathfrak{n}}A_i$, let  $\zeta_p\subset\mathfrak{M}\cap\{p_z=z=0\}$ be the planar Kepler orbit on $\mathfrak{M}$ with minimal Reeb period $T_p>0$ and Seifert rotation number $\hat{\rho}_p$. When $\sum_{i=1}^{\mathfrak{n}}A_iB_i\geq\sum_{i=0}^{\mathfrak{n}}A_i$, we have
		$$\hat \rho_{p}\geq T_p^2/\mathrm{vol}(\mathfrak{M},\lambda),$$
	equality holds if and only if $A_0=0$ and $B_i=1,\,i=1,\cdots,\mathfrak{n}$. Moreover, the energy surface $\mathfrak M$ admits infinitely many periodic orbits.
\end{thm}
\begin{rem}
	The equality condition given in Theorem \ref{thm: estimate of rotation number} implies that the system reduces to the reduced spatial Kepler problem. The relation between the contact volume $\mathrm{vol}(\mathfrak M,\lambda)$ and the minimal Reeb period $T_p$ and the Seifert rotation number $\hat \rho_p$ of $\zeta_p$ has been studied in \cite{HQY25}.
\end{rem}
Theorem \ref{thm: estimate of rotation number} implies the following corollary
\begin{cor}
	For Hamiltonian \eqref{aniso Kepler ham}, assuming $A_1>0$ and $-A_1^2<2h\varpi^2<0$, let $\mathfrak{M}=H^{-1}(h)$ and $\zeta_p\subset\mathfrak{M}\cap\{p_z=z=0\}$ be the planar Kepler orbit on $\mathfrak{M}$ with minimal Reeb period $T_p>0$ and Seifert rotation number $\hat{\rho}_p$. When $B_1\geq1$, we have
	$$\hat \rho_{p}\geq T_p^2/\mathrm{vol}(\mathfrak{M},\lambda),$$
	equality holds if and only if $B_1=1$. Moreover, the energy surface $\mathfrak M$ admits infinitely many periodic orbits.
\end{cor}
\begin{rem}
	In \cite{HOQ26}, the authors considered the case $A_1=1$. For $A_1\neq 1$, we can apply a conformal symplectic transformation
	$$\Phi:(p_r,p_z,r,z)\mapsto(p_r,p_z,r/A_1,z/A_1),$$
	which transforms the system into the case with $A_1=1$ and preserves identical dynamics. Combined with Theorem 1.1 in \cite{HOQ26}, we conclude that infinitely many periodic orbits exist on any fixed compact and regular energy surface for all $B_1>0$.
\end{rem}
Theorem \ref{thm: estimate of rotation number} can be applied to the hip-hop $(1+2n)$-body problem, the Hamiltonian \eqref{equ: hiphop ham} of this problem is a special case of \eqref{equ: reduced Ham}, just take $\mathfrak{n}=n+1, A_{0}=a(n), A_{n+1}=m_{0}, B_{n+1}=1$ and $A_{i}=b_{n,i}/4, B_{i}=b_{n,i}^2$ for $i=1,...,n$. The planar Kepler orbit of \eqref{equ: reduced Ham} corresponds to the elliptic relative equilibrium of the regular $(1+2n)$-gon central configuration, which we refer to as the planar ERE. Therefore, we obtain

\begin{cor}\label{coro: estimate of rotation number 1}
	Let $b(n)=\frac{1}{4}\sum_{k=1}^{n}b_{n,k}$, $c(n)=a(n)+b(n)$. Assuming $m_0\geq0,n\geq2$ and $-(c(n)+m_0)^{2}<2h\varpi^2<-a(n)^2$, this makes the energy surface $\mathfrak{M}_1=H_1^{-1}(h)$ compact. Let $\zeta_p\subset\mathfrak{M}_1\cap\{p_z=z=0\}$ be the planar ERE on $\mathfrak{M}_1$ with minimal Reeb period $T_p>0$ and Seifert rotation number $\hat{\rho}_p$. We have $\hat{\rho}_p> T_p^2/\mathrm{vol}(\mathfrak M_1,\lambda_1)$ and there exist infinitely many periodic orbits on compact energy surface $\mathfrak{M}_1$.
\end{cor}

\begin{rem}
	In \cite{Terracini2007}, Terracini and Venturelli proved the existence of a family of hip-hop symmetric solutions for the weak force case. The rotation angle $\alpha$ of these orbits about the $z$-axis over one period takes values within a given interval. However, it is unknown whether these orbits lie on a common energy surface. By Corollary \ref{coro: estimate of rotation number 1}, there exist infinitely many hip-hop symmetric orbits on every compact regular energy surface for the Newtonian potential, yet we cannot determine the total rotation angle about the $z$-axis over one period for any given orbit.
\end{rem}

Another closely related problem is the $n$-pyramidal problem, that is, the motion of $1+n$ point masses in $\rr^3$ under Newton's universal gravitational law. In this system, the latter $n$ identical masses forming a regular $n$-gon are symmetric about a fixed axis, along which the first particle moves.

When $n=2$, we obtain the well-known spatial isosceles three-body problem. This problem has attracted extensive research, see \cite{Ale72,M84,Sit60}. Recently, a symplectic dynamical approach for the spatial isosceles three-body problem was developed in \cite{HLOYS23}, where the authors found a Hopf link which spans an open book decomposition whose pages are annulus-like global surfaces of section. They also proved the existence of infinitely many periodic orbits on compact regular energy surfaces when the mass ratio is large enough and verified the convexity of energy surfaces. In \cite{HLOQS26}, the authors proved the existence of infinitely many periodic orbits on every compact regular energy surface and found a non-trivial twist interval defined by the rotation number of the Euler orbit and the contact volume of energy surface. It encodes the relative winding of periodic orbits. For non-compact energy surfaces, the authors proved the existence of infinitely many periodic orbits and infinitely many parabolic trajectories. In \cite{HQY25}, Hu, Qiao and Yu proved that the $n$-pyramidal problem possesses infinitely many periodic orbits on compact energy surfaces when the mass ratio $\alpha$ small enough. In this paper, we generalize the conclusion to all mass ratios provided that the energy surface remains compact.

Following \cite{HQY25}, the reduced Hamiltonian of the $n$-pyramidal problem can be written as
\begin{equation}\label{equ: reduced Ham 2}
	H_2(p_r,p_z,r,z)=\frac{1}{2}(p_r^2+p_z^2)+\frac{\varpi^2}{2r^2}-\frac{a(n)\alpha}{r}-\frac{1}{\sqrt{r^2+(1+n\alpha)z^2}},
\end{equation}
where $a(n)=\frac{1}{4}\sum\limits_{k=1}\limits^{n-1}\csc(\frac{k\pi}{n})$, $\alpha>0$ is the mass ratio between one of the last $n$ bodies and the first body, and $\varpi$ is the angular momentum with respect to the $z$-axis. One can see that this system is also a special case of \eqref{equ: reduced Ham}, just take $\mathfrak{n}=1, A_{0}=a(n)\alpha, A_{1}=1$ and $B_{1}=(1+n\alpha).$ The planar Kepler orbit of \eqref{equ: reduced Ham} corresponds to the elliptic relative equilibrium of the regular $(1+n)$-gon central configuration, which we refer to as the planar ERE.
\begin{cor}\label{coro: estimate of rotation number 2}
	Assuming $\alpha>0$ and $-(1+a(n)\alpha)^2<2h\varpi^2<-(a(n)\alpha)^2$, this makes the energy surface $\mathfrak{M}_2=H_2^{-1}(h)$ compact. Let $\zeta_p\subset\mathfrak{M}_2\cap\{p_z=z=0\}$ be the planar ERE on $\mathfrak{M}_2$ with minimal Reeb period $T_p>0$ and Seifert rotation number $\hat{\rho}_p$. When $2\leq n\leq472$, we have $\hat{\rho}_p> T_p^2/\mathrm{vol}(\mathfrak M_2,\lambda_2)$ and there exist infinitely many periodic orbits on compact energy surface $\mathfrak{M}_2$.
\end{cor}

The paper is organized as follows. In Section \ref{sec: preliminaries}, we first reduce the hip-hop $(1+2n)$-body problem to system \eqref{equ: hiphop ham} which serves as a special case of system \eqref{equ: reduced Ham}. We then classify the topological structure of energy surfaces and derive a valid criterion for the existence of infinitely many periodic orbits via the CHHL formula for system \eqref{equ: reduced Ham}. In Section \ref{sec: the volume estimate}, we establish the relation between the contact volume of the energy surface and the minimal Reeb period of the planar Kepler orbit. In Section \ref{sec:  estimates of the rotation number}, we evaluate the Seifert rotation number using Maslov-type index theory and complete the proofs of Theorem \ref{thm: estimate of rotation number}, Corollary \ref{coro: estimate of rotation number 1} and Corollary \ref{coro: estimate of rotation number 2}. Finally, we briefly review the Maslov-type index theory in Appendix \ref{sec: App}.



\section{Preliminaries}\label{sec: preliminaries}

\subsection{The hip-hop $(1+2n)$-body problem}\label{sec: $1+2n$}
We consider the motion of $1+2n$ particles in $\mathbb{R}^3$ under Newton's universal gravitational law. There exists an invariant subsystem in which the motion of the last $2n$ particles of equal mass has hip-hop symmetry and the first one is fixed at the center of mass. Let $q_i$ denote the position of the $i$-th body. Assume that $q_1,...,q_{2n}$ have equal mass $m>0$ and hip-hop symmetry, and $q_{0}$ has the mass $mm_0 \geq0$, where $m_0$ is the mass ratio of the central mass to one of the other $2n$ masses. Hip-hop symmetry means that $q_i=Rq_{i-1}$, $i=2,...,2n$, where
\begin{equation}
	R=\begin{pmatrix}
		\cos (\frac{\pi}{n})  & -\sin(\frac{\pi}{n}) & 0 \\[0.7em]
		\sin (\frac{\pi}{n})  & \cos (\frac{\pi}{n}) & 0 \\
		0 & 0 & -1
	\end{pmatrix},
\end{equation}
which represents a rotation combined with a reflection.

Let $q_1=(x,y,z)$ and  $q_{0}=(0,0,0)$, then $\dot {q}_1=(\dot {x},\dot {y},\dot {z})$ and  $\dot {q}_0=(0,0,0)$. The motion of $q_1$  entirely determines the dynamics of the system, and the Lagrangian writes as
\begin{equation}
	\begin{aligned}
		L(q,\dot{q})&=nm(\dot x^{2}+\dot y^{2}+\dot z^{2})+\sum_{k=1}^{2n-1}\frac{nm^{2}}{\sqrt{4(x^{2}+y^{2})\sin^{2} (\frac{k\pi}{2n})+((-1)^k -1)^2 z^{^{2}}}}+\frac{2nm^2m_0}{\sqrt{x^{2}+y^{2}+z^{2}}}.
	\end{aligned}
\end{equation}

Without loss of generality, we assume $m=1$. Let $\hat{L}=L/(2n)$ which is a scaling and denote the same dynamical as $L$,
\begin{equation}
	\begin{aligned}
		\hat{L}(q,\dot{q})=\frac{1}{2}(\dot x^{2}+\dot y^{2}+\dot z^{2})+\frac{1}{2}\sum_{k=1}^{2n-1}\frac{1}{\sqrt{4(x^{2}+y^{2})\sin^{2} (\frac{k\pi}{2n})+((-1)^k -1)^2 z^{^{2}}}}+\frac{m_0}{\sqrt{x^{2}+y^{2}+z^{2}}}.
	\end{aligned}
\end{equation}

Using cylindrical coordinates $(r,\theta,z)\in \mathbb{R}^{+}\times \mathbb{R}/2\pi \mathbb{Z} \times \mathbb{R}$, and by the Legendre transform

$$p_r =\dot r,\ \ \ p_\theta =r^{2}\dot{\theta},\ \ \ p_z =\dot{z},$$	
we can obtain the Hamiltonian
\begin{equation}
	\begin{aligned}
		H(p_r,p_\theta,p_z,r,\theta,z)=\frac{1}{2}(p_r^{2}+\frac{p_\theta ^{2}}{r^{2}}+p_z^{2})-\frac{1}{2}\sum_{k=1}^{2n-1}\frac{1}{\sqrt{4r^{2}\sin^{2} (\frac{k\pi}{2n})+((-1)^k -1)^2 z^{^{2}}}}-\frac{m_0}{\sqrt{r^{2}+z^{2}}}.
	\end{aligned}
\end{equation}

Since $H$ is independent of $\theta$, the angular momentum $p_\theta $ is constant along the trajectories.	We denote $\varpi= p_\theta =r^{2}\dot{\theta}$	and obtain the Hamiltonian \eqref{equ: reduced Ham}

$$H_{1}(p_r,p_z,r,z)=\frac{1}{2}(p_r^{2}+p_z^{2})+V_1(r,z),$$
with the corresponding potential

$$V_1(r,z)=\frac{\varpi^{2}}{2r^{2}}-\frac{1}{2}\sum_{k=1}^{2n-1}\dfrac{1}{\sqrt{4r^{2}\sin^{2} (\frac{k\pi}{2n})+((-1)^k -1)^2 z^{^{2}}}}-\frac{m_0}{\sqrt{r^{2}+z^{2}}}.$$	



To simplify the expression, we define
\begin{equation}
	a(n)=\frac{1}{4}\sum_{k=1}^{n-1}a_{n,k} \text{ with } a_{n,k}=\csc\left(\frac{k\pi}{n}\right) \text{ and } b(n)=\frac{1}{4}\sum_{k=1}^{n}b_{n,k} \text{ with } b_{n,k}=\csc\left(\frac{(2k-1)\pi}{2n}\right),
\end{equation}
and $V_1(r,z)$ will write as

$$V_1(r,z)=\frac{\varpi^{2}}{2r^{2}}-\frac{a(n)}{r}-\frac{1}{4}\sum_{k=1}^{n}\frac{b_{n,k}}{\sqrt{r^2+b_{n,k}^2z^2}}-\frac{m_0}{\sqrt{r^{2}+z^{2}}}.$$

\subsection{Topology of the energy surfaces}\label{sec: topology}

We consider the topology of the energy surfaces for the reduced Gutzwiller-type anisotropic kepler problem.

\begin{prop}\label{topo of energy surfaces 1}
	Assume $A_0\geq0$ and $A_i,B_i>0,\,i=1,\cdots,\mathfrak{n}$, $C=\sum_{i=0}^{\mathfrak{n}}A_i$. Let $\mathfrak{M}=H^{-1}(h)$ be the energy surface of the reduced Hamiltonian \eqref{equ: reduced Ham}. The following statements hold:
	\begin{itemize}
		\item[(i)] If $-C^2<2h\varpi^2<-A_0^2$, then $\mathfrak{M}$ is a compact manifold homeomorphic to $\mathbb{S}^{3}$.
		\item[(ii)] If $2h\varpi^2=-C^2$, then $\mathfrak{M}=\{(0,0,\varpi^2/C,0)\}$.
		\item[(iii)] If $2h\varpi^2<-C^2$, then $\mathfrak{M}=\emptyset$.
		\item[(iv)] If $2h\varpi^2\geq-A_0^2$, then $\mathfrak{M}$ is unbounded in $z$ direction.
	\end{itemize}
\end{prop}
\begin{proof}
	The proof is similar to \cite[Proposition 2.1]{HQY25} and \cite[Proposition 2.1]{HOQ26}. First, by solving $$\nabla H=\left(p_{r},p_{z},-\frac{\varpi^{2}}{r^{3}}+\frac{A_0}{r^2}+\sum_{i=1}^{\mathfrak{n}}\frac{A_ir}{(r^2+B_iz^2)^{3/2}},\sum_{i=1}^{\mathfrak{n}}\frac{A_iB_iz}{(r^2+B_iz^2)^{3/2}}\right)^T=0,$$
	we know $(0,0,\varpi^2/C,0)$ is the unique critical point of $H$. Meanwhile, the Hessian matrix at $(0,0,\varpi^2/C,0)$ is
	$$\nabla^{2}H(0,0,\varpi^2/C,0)=\text{diag}(1,1,C^4/\varpi^6, C^3D/\varpi^6),$$
	where $D=\sum_{i=1}^{\mathfrak{n}}A_iB_i$, which is positive definite. Hence, the unique critical
	point is a non-degenerate global minimum and $H(0,0,\varpi^2/C,0)=-C^2/(2\varpi^2)$. Then $(ii), (iii)$ are proved. Since $H$ is increasing with respect to $z^2$, when $h\geq\liminf_{z\rightarrow\infty}H=-A_0^2/(2\varpi^2)$,  $\mathfrak{M}$ is unbounded in $z$ direction. This proves $(iv)$.
	For $(i)$, when $-C^2<2h\varpi^2<-A_0^2$, by above, $\mathfrak{M}$ is bounded in $z$ direction. Since $H$ increases with respect to $z^2$, we consider the case $z=0$, which attains the maximum range of $r$, by direct computation, $r\in[\frac{C(1-\fe)}{-2h},\frac{C(1+\fe)}{-2h}]$ is bounded. So it is compact. By Morse theory, as $h$ increases from
	below $-C^2/(2\varpi^2)$ to slightly above $-C^2/(2\varpi^2)$, $\mathfrak{M}$
	changes from the empty set to a sphere-like regular hypersurface. Since $\mathfrak{M}$ is
	bounded for every $h$ in that interval, it is a sphere-like hypersurface.
	

\end{proof}

\subsection{Reeb flows and Rotation numbers}\label{sec: Reeb flow and Rotation number}
Let $(M,\xi=\ker\lambda)$ be a contact three-sphere with contact form $\lambda$, which is a $1$-form so that $\lambda\wedge d\lambda$ becomes a volume form on $M$. The Reeb vector field $R$ is determined by $\lambda(R)=1,d\lambda(R,\cdot)\equiv0$ and preserves the contact structure $\xi:=\ker \lambda$. $\xi$ is co-oriented by $\lambda$, on which there is a symplectic form $d\lambda$. The flow $\{\phi_t\}_{t\in \mathbb R}$ of $R$ is called the Reeb flow of $(M,\xi=\ker\lambda)$. A periodic orbit (or Reeb orbit) of $R$ is a map $\zeta:[0,T]\rightarrow M$, which is a periodic trajectory of $\phi_t$ with period (or Reeb period) $T>0$. We call $\zeta$ simple if $T$ is the minimal period of $\zeta$. Denote $\mathcal P$ the collection of Reeb orbits. Each simple element $\zeta\in \mathcal P$ bounds an embedded surface $\mathcal S$ in $M$, called Seifert surface, whose $d\lambda$-area is equal to $T$. We call $\zeta$ unknotted if $\mathcal S$ is a disk. Identifying the Reeb orbits, which differ by a time shift. Let $\tau$ be a trivialization of $\xi$ over $\zeta$ in the distinguished homotopy class determined by $\mathcal S$, for which the outward normal vector field to $\mathcal S$ has zero winding number along $\zeta$. Under the trivialization $\tau$, the transverse flow $d\phi_t:\xi|_{\zeta(0)}\rightarrow \xi|_{\zeta(t)}$ determines a symplectic path $\gamma(t),t\in[0,T]$, in $\mathrm{Sp}(2)$. Define the Seifert rotation number of $\zeta$ as
$$
\hat \rho(\zeta):=\hat i(\zeta)/2,
$$
where $\hat i(\zeta)$ is the mean index of the sympletic path $\gamma$, see Section \ref{sec: App}. Define the linearized return map of $\zeta$ as $d\phi_T: \xi_{\zeta(0)}\rightarrow \xi_{\zeta(T)}$. A Reeb orbit $\zeta$ is a called elliptic if the two eigenvalues of $d\phi_T$ lies in the unit circle $\mathbf U\subset \mathbb C$. Otherwise, we call $\zeta$ a hyperbolic Reeb orbit. A Reeb orbit $\zeta$ is called irrational elliptic if $\hat \rho(\zeta)$ is an irrational number.

If $(M,\xi=\ker\lambda)$ is a contact-type three-sphere in $(\mathbb R^4,\omega_0)$ with the standard symplectic form $\omega_0=dy_1\wedge dx_1+dy_2\wedge dx_2$ satisfying $d\lambda=\omega_0$, we can choose a $d\lambda$-positive frame $\{X_1,X_2\}\subset TM$,
$$\begin{aligned}
X_1&=n_4\partial_{y_1}-n_3\partial_{y_2}+n_2\partial_{x_1}-n_1\partial_{x_2},\\
X_2&=-n_2\partial_{y_1}+n_1\partial_{y_2}+n_4\partial_{x_1}-n_3\partial_{x_2}.
\end{aligned}$$
where $n:=n_1\partial_{y_1}+n_2\partial_{y_2}+n_3\partial_{x_1}+n_4\partial_{x_2}$ is the outward normal vector field on $M$. In particular, $X_1,X_2$ are transverse to the Reeb vector field $R$. Then for every $\zeta\in \mathcal P$, we obtain a trivialization $\tau_0$ of $\xi$ over $\zeta$ using $\{X_1,X_2\}$ and the projection $TM\rightarrow \xi$ along $R$ and it determines a symplectic path $\gamma_0(t),t\in[0,T]$, in $\Sp(2)$. Define another rotation number $\rho(\zeta):=\hat i(\gamma_0)/2$ as used in \cite{HLOYS23}, which coincide with usual rotation number by linearied equation along $\zeta$. For an unknotted Reeb orbit $\zeta\in \mathcal P$, since the normal vector field to $\mathcal S$ along $\zeta$ always admits winding number $-1$ in $\tau_0$, we have the following relation
$$
\rho(\zeta)=\hat \rho(\zeta)+1.
$$

For a general tight three-sphere $(M,\xi=\ker\lambda)$, Cristofaro-Gardiner, Hryniewicz, Hutchings and Liu \cite{CHHL23,CHHL26} proved the two-or-infinity conjecture and obtained an interesting formula (the CHHL formula) for the relation among contact volume, minimal Reeb period, and Seifert rotation number in the case of exactly two simple Reeb orbits. 
\begin{thm}[\cite{CHHL23} Theorems 1.2 and 1,5; \cite{CHHL26} Theorem 1.1]\label{thm: CHHL23}
Let $(M,\xi=\ker\lambda)$ be a tight three-sphere, then $\lambda$ has either two or infinitely many simple Reeb orbits. When $\lambda$ has exactly two simple Reeb orbits $\zeta_{1},\zeta_{2}$ with minimal Reeb period $T_1,T_2$. Let $\hat\rho_{i}=\hat\rho(\zeta_i)$ denotes the Seifert rotation number of $\zeta_{i},\ i=1,2$, and let $\mathrm{vol}(M,\lambda)$ be the contact volume of $M$, then
\begin{equation*}\label{eq:Kepler equal}
\mathrm{vol}(M,\lambda)=T_{1}^{2}/\hat\rho_{1}=T_{2}^{2}/\hat\rho_{2}.
\end{equation*}
Moreover, both $\zeta_{1}$ and $\zeta_{2}$ are irrationally elliptic.
\end{thm}
From this theorem, we obtain the following useful corollary.


\begin{cor}\label{cor: infini peri orbit1}
	Assume $A_0\geq0$ and $A_i,B_i>0,\,i=1,\cdots,\mathfrak{n}$. Let $\mathfrak{M}=H^{-1}(h)$ be the energy surface of the reduced Hamiltonian \eqref{equ: reduced Ham} and  $\zeta_p\subset \mathfrak M\cap \{p_z=z=0\}$ be the planar Kepler orbit on $\mathfrak M$ with minimal Reeb period $T_p$ and rotation number $\hat{\rho}_p=\hat\rho(\zeta_p)$. If
	$$
	\emph{vol}(\mathfrak{M},\lambda)\neq T_{p}^2/\hat{\rho}_p \text{ or } \hat{\rho}_p\in \mathbb{Q},
	$$
	then the system \eqref{equ: reduced Ham sys} has infinitely many periodic orbits on $\mathfrak{M}$.
\end{cor}

When $A_0=0$ and $\mathfrak{n}=1$, the estimate of the rotation number $\hat \rho_p:=\hat \rho(\zeta_p)$ using the contact volume of the energy surface $\mathfrak M$ and the minimal Reeb period of planar Kepler orbit $\zeta_p$ for $B_1\in(0,1]$ has been obtained in \cite{HOQ26}.
\begin{thm}\label{thm: estimate of rotation number 2.2}
	(\cite{HOQ26})When $A_0=0$ and $\mathfrak{n}=1$, assume $-A_1^2<2h\varpi^2<0$ and let $\zeta_p\subset\{p_z=z=0\}$ be the planar Kepler orbit on $\mathfrak M$ with the minimal Reeb period $T_p>0$ and rotation number $\hat \rho_p$. Then
	\begin{itemize}
		\item[(i)] for $B_1\in (0,1/4)$, we have $\hat \rho_{p}>T_p^2/\mathrm{vol}(\mathfrak{M},\lambda)=\sqrt{B_1}$.
		\item[(ii)] for $B_1\in (1/4,1)$, we have $\hat \rho_{p}<T_p^2/\mathrm{vol}(\mathfrak{M},\lambda)=\sqrt{B_1}$.
		\item[(iii)] for $B_1\in \{1/4,1\}$, we have $\hat \rho_{p}=T_p^2/\mathrm{vol}(\mathfrak{M},\lambda)\in\{1/2,1\}$.
	\end{itemize}
	Moreover, the energy surface $\mathfrak M$ admits infinitely many periodic orbits.
\end{thm}


\section{The volume estimates}\label{sec: the volume estimate}

In this section, we derive the relation between the contact volume of the energy surface and the minimal Reeb period of the planar Kepler orbit. For the reduced Gutzwiller-type anisotropic Kepler problem, we regard $\zeta_p$ as a Reeb orbit of the Reeb vector field $R$. The contact volume and the minimal Reeb period of $\zeta_p$ are defined respectively as $\mathrm{vol}(\mathfrak M,\lambda):=\int_{\mathfrak M} \lambda \wedge d\lambda$ and  $T_p:=\int_{\zeta_p}\lambda=\int_{\Upsilon}dp_r\wedge dr$. When $A_0=0$ and $\mathfrak{n}=1$, the relation between $\mathrm{vol}(\mathfrak M,\lambda)$ and $T_p$ has been established in \cite{HQY25} for $B_1=1$ and \cite{HOQ26} for $B_1>0$.
\begin{thm}[\cite{HQY25} Lemma 2.1; \cite{HOQ26} Proposition 2.3]\label{thm: vol of M}
	When $A_0=0$ and $\mathfrak{n}=1$, assume $-A_1^2<2h\varpi^2<0$, $\fe=(1+2h\varpi^2/A_1^2)^{1/2}$, we have
	$$
	\mathrm{vol}(\mathfrak{M},\lambda)=\frac{T_p^2}{\sqrt{B_1}},\quad T_p=2\pi\varpi\left(\frac{1}{\sqrt{1-\fe^2}}-1\right).
	$$
\end{thm}

To estimate the contact volume in the general case, we consider the following anisotropic Kepler problem with $\mathfrak{n}=1$:

\begin{equation}\label{equ;anso compa}
	\hat{H}_0 (p_r,p_z,r,z)  =\frac{1}{2}(p_r ^{2}+p_z ^{2})+\hat{V}_0 (r,z),
\end{equation}	
with
\begin{equation}
	\hat{V}_0(r,z)=\frac{\varpi^{2}}{2r^{2}}-\frac{C}{\sqrt{r^{2}+(D/C) z^{2}}},
\end{equation}
where $$D=\sum_{i=1}^{\mathfrak{n}}A_iB_i\text{  and  } C=\sum_{i=0}^{\mathfrak{n}}A_i.
$$

Denote $\hat{\mathfrak{M}}_0=\hat{H}_0 ^{-1}(h)$, we can easily verify that the planar Kepler orbit $\hat \zeta_p$ of system \eqref{equ;anso compa} coincide with $\zeta_p$ of system \eqref{equ: reduced Ham sys} and they have the same minimal Reeb period $T_p$. For the contact volume of energy surface, we obtain the following estimate.
\begin{prop}\label{prop, vol1 vol2}
	$\mathrm{vol} (\mathfrak{M},\lambda)\geq\mathrm{vol}(\hat{\mathfrak{M}}_0,\hat{\lambda}_0)$, equality holds if and only if $A_0=0$ and $B_1=B_2=\cdots=B_\mathfrak{n}$.
\end{prop}
\begin{proof}
	When $A_0=0$ and $B_1=B_2=\cdots=B_\mathfrak{n}$, the result is obvious. Then we assume $A_0>0$ or $B_i$ are not all equal. We can easily verify that the function $f(x)=(x+r^2)^{-1/2}$ is strictly convex for $x>-r^2$, then using Jensen inequality, we have for $\forall \lambda_i\geq0,\,\sum_{i=0}^{\mathfrak{n}}\lambda_i=1$,
	$$f\left(\sum_{i=0}^{\mathfrak{n}}\lambda_ix_i\right)\leq\sum_{i=0}^{\mathfrak{n}}\lambda_if(x_i),
	$$equality holds if and only if $x_0=x_1=\cdots=x_\mathfrak{n}$.
	Let $\lambda_i=A_i/C$, $x_0=0$ and $x_i=B_iz^2,\,i=1,\cdots,\mathfrak{n}$, we obtain
	$$\frac{A_0}{r}+\sum_{i=1}^{\mathfrak{n}}A_i(r^2+B_iz^2)^{-1/2}>\left(\sum_{i=0}^{\mathfrak{n}}A_i\right)\cdot\left(r^2+\frac{\sum_{i=1}^{\mathfrak{n}}A_iB_i}{\sum_{i=0}^{\mathfrak{n}}A_i}z^2\right)^{-1/2},
	$$therefore $V(r,z)<\hat{V}_0 (r,z)$.
	
	Using Stokes formular, the contact volume $\mathrm{vol}(\mathfrak{M},\lambda)=2\int_{\mathfrak{B}}dp_r\wedge dr\wedge dp_z\wedge dz$ and $\mathrm{vol}(\hat{\mathfrak{M}}_0,\hat{\lambda}_0)=2\int_{\hat{\mathfrak{B}}_0}dp_r\wedge dr\wedge dp_z\wedge dz$, where $\mathfrak{B}=\{(p_r,p_z,r,z)|H\leq h\}$ and $\hat{\mathfrak{B}}_0=\{(p_r,p_z,r,z)|\hat{H}_0\leq h\}$ are bounded by $\mathfrak{M}$ and $\hat{\mathfrak{M}}_0$ respectively. The above estimate implies $\hat{\mathfrak{B}}_0\subset\mathfrak{B}$, so $\mathrm{vol} (\mathfrak{M},\lambda)>\mathrm{vol}(\hat{\mathfrak{M}}_0,\hat{\lambda}_0)$.
\end{proof}
\section{Estimates of the rotation number}\label{sec:  estimates of the rotation number}

\subsection{The planar Kepler orbit} \label{sec: the planar Kepler orbit}

Assume $-C^{2}<2h\varpi^2<-A_0^2$. Recall that $\zeta_p=(\dot r_p,0,r_p,0)\subset \mathfrak M\cap \{p_z=z=0\}$ is the planar Kepler orbit of the system \eqref{equ: reduced Ham sys} on the energy surface $\mathfrak M$ of the Hamiltonian $H$ in \eqref{equ: reduced Ham}, which satisfies the equation $\ddot{r}_{p}=\varpi^2/r_{p}^3-C/r_p^2$. We use a similar technical method to compute the rotation number of $\zeta_p$ in \cite{HLOYS23}, where the authors dealt with the Euler orbit. This method were also used in \cite{HOQ26,HLOQS26}. Let $\bar T_p$ be the minimal period of $\zeta_p$, which is not same as the minimal Reeb period $T_{p}$. The linearized flow of $\zeta_p$ is given by
$$
\dot{\gamma}_{p}(t)=J_4\mathcal{B}(t)\gamma_{p}(t),\ \ \gamma_{p}(0)=I_2.
$$	
where $\mathcal{B}(t):=\mathrm{diag}(1,1,(3\varpi^2-2C r_p)/r_p^{4},D/r_p^{3})$ and $J_{2n}$ denotes the standard symplectic matrix in $\mathrm{Sp}(2n)$. This is a decoupled system with the following two subsystems
\begin{equation}\label{equ: linear equ1}
\dot{\gamma}_{1}(t)=J_2\left(\begin{array}{cc}
	1 & 0\\
	0 & (3\varpi^2-2C r_p(t))/r_p^4(t)
\end{array}\right)\gamma_{1}(t),\quad \gamma_1(0)=I_2,
\end{equation}
and
\begin{equation}\label{equ: linear equ2}
\dot{\gamma}_{2}(t)=J_2\left(\begin{array}{cc}
	1 & 0\\
	0 & D/r_p^3(t)
\end{array}\right)\gamma_{2}(t),\quad \gamma_2(0)=I_2,
\end{equation}
where $\gamma_1,\gamma_2$ denote the associated fundamental solutions of these two subsystems. In particular, the tangent vector field $\xi_1(t)=c\cdot (\ddot{r}_p(t), \dot r_p(t)),c\in \mathbb R_+$, of $\zeta_p$ is a periodic solution of \eqref{equ: linear equ1}. Choose a $c>0$ and a time shift so that $\xi_1(0)=(1,0)^T$ and then $\xi_1$ becomes the first column of $\gamma_{1}$. Let $\Lambda_D:=\mathbb{R}\oplus\{0\}$ be a Lagrangian subspace of $\rr^2$. Since $\mathcal B(t)>0,\forall t\in \mathbb R$, we compute from \eqref{mean index2} that
$$
\hat{i}(\gamma_{1})=\lim_{m\to +\infty}\frac{1}{m}\sum_{0<t\leq m\bar T_p} \operatorname{dim} \gamma_{1}\left(\hat{t}\right) \Lambda_{D} \cap \Lambda_{D}=2.
$$
where $\sum_{0<t\leq m\bar T_p} \operatorname{dim}(\gamma_{1}(t)\Lambda_{D} \cap \Lambda_{D})=\sharp\{t:\dot{r}_{p}(t)=0,t\in(0,m\bar T_p]\}
=2m-1$ for every $m\in \mathbb Z_+$.

In order to compute the mean index $\hat{i}(\gamma_{2})$, we consider $t=t(\theta)$ based on the relation $\dot{\theta}(t)=\varpi/r^2_{p}(t)$ and indicate $'=d/d\theta$, then \eqref{equ: linear equ2} can be rephrased as
$$\gamma_{2}'(\theta)=J_2\left(\begin{array}{cc}
	r_p^2(\theta)/\varpi & 0\\
	0 & D/(\varpi r_p(\theta))
\end{array}\right)\gamma_{2}(\theta),
$$
where $r_p(\theta)=\frac{\varpi^2/C}{1+\mathfrak e \cos\theta}$. Using a time-dependent linear symplectic transformation
$$
\hat\gamma_{p}(\theta):=\mathcal{R}(\theta)\gamma_{2}(\theta),\quad \theta\in \mathbb R/2\pi\mathbb Z,
$$
where
$$\mathcal{R}(\theta):=\left(\begin{array}{cc}
	\frac{r_p}{\sqrt{\varpi}} & \frac{-\sqrt{\varpi}r_p'}{r_p^2}\\
	0 &\frac{\sqrt{\varpi}}{r_p}
\end{array}\right)=\left(\begin{array}{cc}
\frac{r_p}{\sqrt{\varpi}} & -\frac{\sqrt{\varpi}}{r_p}\frac{\fe\sin\theta}{1+\fe\cos\theta}\\
0 &\frac{\sqrt{\varpi}}{r_p}
\end{array}\right)=\left(\begin{array}{cc}
\frac{r_p}{\sqrt{\varpi}} & -\frac{C\fe\sin\theta}{\varpi^{3/2}}\\
0 &\frac{\sqrt{\varpi}}{r_p}
\end{array}\right).
$$
we further obtain the Hill stability equation below as \eqref{equ: Hill stability equation}
\begin{equation}\label{equ: linearized equation rp}
	\hat\gamma'_{p}(\theta)=J_2\left(\begin{array}{cc}
		1 & 0\\
		0 & 1+\frac{\beta}{1+\fe\cos\theta}
	\end{array}\right)\hat \gamma_{p}(\theta),
\end{equation}
where $\beta=D/C-1$. Since the Maslov index of the symplectic loop is $\mathcal R(\theta),\theta\in \mathbb R/2\pi \mathbb Z$, is zero, we have $\hat{i}(\gamma_{2})=\hat{i}(\hat \gamma_{p})$ and $\hat{i}(\gamma_{p})=\hat{i}(\gamma_{1})+\hat{i}(\gamma_{2})=2+\hat{i}(\hat \gamma_p)$.
Moreover, since the trivialization $\tau$ is induced by the normal vector field $\partial_{p_z}$ of the disk $\Sigma$, we obtain the rotation numbers of $\zeta_p$ as
\begin{equation*}\label{equ: rotation number}
\hat \rho_{p}=\frac{\hat{i}(\hat \gamma_{p})}{2}\quad \text{and}\quad
\rho_{p}=\frac{\hat{i}(\gamma_{p})}{2}=1+\hat \rho_p,\quad \forall \beta\in(-1,+\infty).
\end{equation*}
Here $\hat\rho_p:=\hat\rho(\zeta_p)$ and $\rho_p:=\rho(\zeta_p)$ are introduced in both Sections \ref{sec: Reeb flow and Rotation number} and \ref{sec: App}.







\subsection{The degenerate curves }\label{sec: the degenerate curves }
Under the above analysis, we consider the following self-adjoint Fredholm operator on $L^2([0,2\pi], \mathbb{C})$
$$
\mathcal{A}=\mathcal A(\beta,\mathfrak e):=-\frac{d^2}{d\theta^2}-1-\frac{\beta}{1+\fe\cos \theta},\quad \forall\  1+\beta>0,\ 0<\fe<1,
$$
with domain $D(\omega,2\pi):=\left\{W^{2,2}([0,2\pi], \mathbb{C}^n): y(2\pi)=\omega y(0), \dot{y}(2\pi)=\omega \dot{y}(0)\right\}$ for any $\omega\in\mathbf{U}$. We call the condition in $D(\omega,2\pi)$ as the $\omega$-boundary condition. Then $\mathcal A x=0$ is the Hill stability equation \eqref{equ: Hill stability equation}. Let $\nu_\omega(\mathcal A):=\dim_{\mathbb C}\ker \mathcal A$ denote the nullity of $\mathcal A$ and let $m^-_\omega(\mathcal{A})$ denote the Morse index of $\mathcal A$, i.e. the total multiplicity of the negative eigenvalues. From the relations \eqref{index equ} in Section \ref{sec: App}, we obtain
$$
m^-_{\omega}(\mathcal{A})=i_{\omega}(\gamma_{\beta,\mathfrak{e}}), \quad \nu_{\omega}(\mathcal{A})=\nu_{\omega}(\gamma_{\beta,\mathfrak{e}}),
$$
where $\gamma_{\beta,\mathfrak{e}}:=\hat \gamma_p$ is the fundamental solution of equation \eqref{equ: linearized equation rp}, $i_\omega(\gamma_{\beta,\mathfrak{e}})$ and $\nu_\omega(\gamma_{\beta,\mathfrak{e}})$ denote the $\omega$-index and $\omega$-nullity of $\gamma_{\beta,\mathfrak e}$, respectively. We say $\mathcal{A}$ is $\omega$-degenerate if $v_{\omega}(\mathcal{A})\neq 0$. Let $\beta_j=\beta_j(\fe,\omega), j\in\mathbb{Z}_{+}$, denote the $j$-th eigenvalue of the self-adjoint operator $(1+\fe\cos\theta)(-d^2/d\theta^2-1)$ on $D(\omega,2\pi)$ counting the multiplicity so that
$\beta_1\leq\beta_2\leq\cdots\leq\beta_j\leq\cdots$. In particular, $\beta_j(\fe,\omega)=\beta_j(\fe,\bar{\omega}),\forall \omega\in \mathbf{U}$ and $\mathcal{A}(\beta_j(\mathfrak{e},\omega),\mathfrak{e})$ is $\omega$-degenerate for every $j\in \mathbb Z_+$.

Define $\Gamma_j(\omega):=\{(\beta_j(\fe,\omega),\fe): 0\leq \fe<1\},\forall j\in \mathbb Z_+$, as the $j$-th $\omega$-degenerate curve of the Hill stability equation \eqref{equ: Hill stability equation}. We summarize the properties of the $\omega$-index and degenerate curves as below.
\begin{thm}[\cite{HLOYS23,HOT23}]\label{thm: prop of deg curve}
Assume $0\leq \mathfrak{e}<1,1+\beta>0$ and $\omega\in\mathbf{U}$. The following statements hold for the $\omega$-index:
\begin{itemize}
\item[$(a)$] The index $i_{\omega}(\gamma_{\beta,\mathfrak{e}})$ is non-decreasing with $\beta$ for every $1+\beta>0$. If $\beta=-1$, then $i_{\omega}(\gamma_{-1,\mathfrak{e}})=0$.
\item[$(b)$] The map $\beta\mapsto i_{\omega}(\gamma_{\beta,\mathfrak{e}})$ only increases at $\beta=\beta_{j}(\mathfrak{e},\omega),\forall j\in\mathbb{Z}_{+}$. Moreover, $\lim_{\epsilon\rightarrow0}i_{\omega}(\gamma_{\beta+\epsilon,\mathfrak{e}})
=i_{\omega}(\gamma_{\beta,\mathfrak{e}})+\nu_{\omega}(\gamma_{\beta,\mathfrak{e}})$.
\item[$(c)$] For every $k\in\mathbb{Z}_{+}, 0<\mathfrak{e}<1, \omega\in\mathbf{U}\setminus\{1, -1\}$, we have
$$\begin{aligned}
-1&=\beta_{1}(\mathfrak{e},1)<\beta_{1}(\mathfrak{e},\omega)<\beta_{1}(\mathfrak{e},-1)<\beta_{2}(\mathfrak{e},-1)<\beta_{2}(\mathfrak{e},\omega)\\
<0=\beta_{2}(\mathfrak{e},1)&=\beta_{3}(\mathfrak{e},1)<\beta_{3}(\mathfrak{e},\omega)<\beta_{3}(\mathfrak{e},-1)<\beta_{4}(\mathfrak{e},-1)
<\beta_{4}(\mathfrak{e},\omega)\\
<\beta_{4}(\mathfrak{e},1)&=\beta_{5}(\mathfrak{e},1)<\cdots\\
<\beta_{2k}(\mathfrak{e},1)&=\beta_{2k+1}(\mathfrak{e},1)<\beta_{2k+1}(\mathfrak{e},\omega)<\beta_{2k+1}(\mathfrak{e},-1)
<\beta_{2k+2}(\mathfrak{e},-1)<\beta_{2k+2}(\mathfrak{e},\omega)\\
<\beta_{2k+2}(\mathfrak{e},1)&=\beta_{2k+3}(\mathfrak{e},1)<\cdots.
\end{aligned}$$
For every $k\in\mathbb{Z}_+$ and $\mathfrak{e}=0$, we have
\begin{gather*}
\beta_{2k}(0,1)=\beta_{2k+1}(0,1)=k^2-1,\quad \beta_{2k-1}(0,-1)=
\beta_{2k}(0,-1)=(k-1/2)^2-1,\\
\beta_{2k-1}(0,e^{2\pi i\nu})=(k-1+\nu)^2-1,\quad \beta_{2k}(0,e^{2\pi i\nu})=(k-\nu)^2-1,\quad \forall \nu\in(0,1/2).
\end{gather*}

\item[$(d)$] For every $\omega\in\mathbf{U},j\in \mathbb Z_+$, $\Gamma_{j}(\omega)$ is a real analytic curve, i.e. $\beta_{j}(\mathfrak{e},\omega)$ is an analytic function of $\mathfrak{e}$ on $(0,1)$. Moreover, $\partial_{\mathfrak{e}}\beta_{j}(0,\omega)=0$ for every $\beta_{j}(0,\omega)\neq -3/4=\beta_{1}(0,-1)=\beta_{2}(0,-1)$ and $-\partial_{\mathfrak{e}}\beta_{1}(0,-1)=\partial_{\mathfrak{e}}\beta_{2}(0,-1)=3/8$.

\item[$(e)$] For every $\omega\in\mathbf{U}$, we have $\lim_{e\rightarrow1}\beta_{1}(\mathfrak{e},\omega)=-1$, $ \lim_{\mathfrak{e}\rightarrow1}\beta_{j}(\mathfrak{e},\omega)=0,j=2,3$, and
$$\lim_{\mathfrak{e}\rightarrow1}\beta_{n}(\mathfrak{e},\omega)=1/8,\quad \forall n\geq 4.$$

\end{itemize}
\end{thm}

To be more intuitive, we provide the numerical figure of the degenerate curves $\{\Gamma_j(\pm 1)\}_{j\in \mathbb Z_+}$ as in Figure \ref{picture of Birfurcation}.
\begin{figure}[hbpt]
\centering
\includegraphics[width=0.8\textwidth]{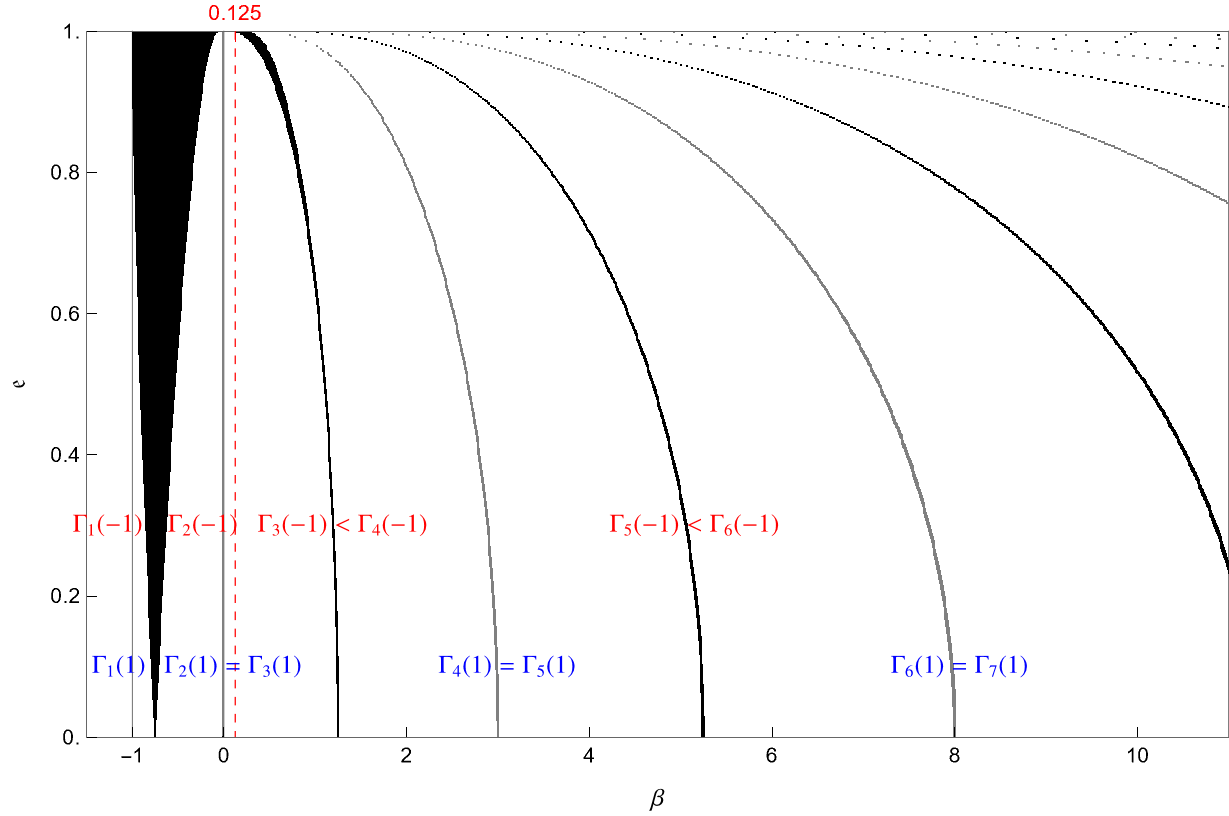}
\caption{The degenerate curves $\Gamma_{j}(1),\Gamma_{j}(-1),j\in \mathbb Z_+$ of the Hill stability equation.}
\label{picture of Birfurcation}
\end{figure}

When  $\beta\in(-1,0]$, the property of the Seifert ratation number has been studied in \cite{HOQ26} and summarised in Theorem \ref{thm: estimate of rotation number 2.2}. When  $\beta>0$, the authors of \cite{HLOQS26} investigated the location of degenerate curves and obtained the following theorem.

\begin{thm}[\cite{HLOQS26} Theorem 2.2 and 2.9]\label{thm: comparison mean index 3}
	For every $\fe\in (0,1),\beta\in(0,6/\fe+3],$ we have
	
$$\hat{\rho}_{p,\beta, \fe} > \hat{\rho}_{p,\beta, 0} =  \sqrt{1+\beta}.$$
\end{thm}

Actually, the above inequality is also valid for $\beta\in[6/\fe+3,+\infty)$.  
\begin{thm}\label{thm: comparison mean index 4}
For every $\mathfrak{e}\in(0, 1)$ and $\beta\in(\beta_*(\fe),+\infty)$, we have $\hat \rho_{p,\beta,\fe}>\hat\rho_{p,\beta,0}=\sqrt{1+\beta}$.
where
\bea\label{Upsil}
\beta_*(\fe):=
\left\{
\begin{aligned}
&\frac{3\mathfrak{e}-4}{4}+\frac{3\mathfrak{e}\pi^2}{\mathcal{T}^2_{\mathfrak{e}}-4\pi^2}, & \hbox{\emph{if} $\fe\in(0,2/3]$,} \\
&\frac{1-3\sqrt{5}\sqrt{1-\mathfrak{e}^2}}{8}+
\frac{3\pi^2(3-\sqrt{5}\sqrt{1-\mathfrak{e}^2})}{2(\mathcal{T}^2_{\mathfrak{e}}-4\pi^2)}, & \hbox{\emph{if} $\fe\in[2/3,1)$,}
\end{aligned}
\right.
\eea
is continuous on $(0,1)$ and $\mathcal{T}_{\mathfrak{e}}=\int_{0}^{2\pi}(1+\mathfrak{e}\cos\theta)^{-\frac{1}{2}}d\theta$. Moreover, we have $\beta_*(\mathfrak e)<6/\fe+3$, see Figure~\ref{fig:Upsilon}.
\end{thm}
\begin{figure}[hpbt]
\centering
\includegraphics[width=0.6\textwidth]{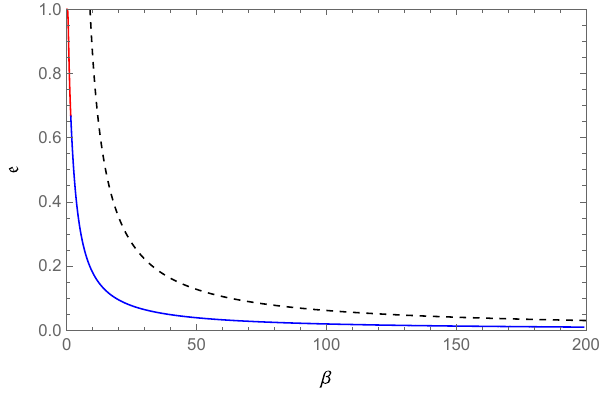}
\caption{The graphs of $\beta_*(\fe)$ (red for $2/3\leq \mathfrak e<1$, blue for $0<\mathfrak e\leq 2/3$) and $6/\mathfrak e+3$ (dashed)}
\label{fig:Upsilon}
\end{figure}
To prove this theorem, we need to use a blow-up technique, which was introduced in \cite{MSS06} and be used in \cite{HO16} to develop the collision index theory and study the linear stability of the elliptic relative equilibria near collisions. Let
$$
q=(1+\mathfrak{e}\cos\theta)^{\frac{1}{2}},\quad Q(\theta)=-2q'(\theta),\quad d\theta=q(\theta)d\tau.
$$
One can directly check that
\bea \label{equ: q,Q equation}
\frac{dQ}{d\tau}=\frac{1}{2}Q^{2}+q^{2}-1,\quad \frac{dq}{d\tau}=-\frac{1}{2}q Q,
\eea
and the following lemma hold. Notice that
\begin{equation}\label{equ: first integral E}
E(Q,q):=q^{2}(Q^{2}/2+q^{2}/2-1)=(\mathfrak{e}^{2}-1)/2,
\end{equation}
is a first integral of the system \eqref{equ: q,Q equation} and
$\mathcal{T}_{\mathfrak{e}}=\tau(2\pi)$ is the period of new time variable $\tau$.
\begin{lem}\label{lem: Perod}
Let $\mathcal{T}_{\mathfrak{e}}=\tau(2\pi)$ be the new time period above. The following assertions hold:
\begin{itemize}
\item[(i)] $\mathcal{T}_{0}=2\pi$ and $\lim_{\mathfrak{e}\rightarrow 1^+}\mathcal{T}_{\mathfrak{e}}=+\infty$.
\item[(ii)] $\frac{\partial\mathcal{T}_{\mathfrak{e}}}{\partial\mathfrak{e}}\big{|}_{\mathfrak{e}=0}=0$ and
$\frac{\partial\mathcal{T}_{\mathfrak{e}}}{\partial\mathfrak{e}}>0$, $\mathcal{T}_{\mathfrak{e}}>\frac{4\pi}{\sqrt{4-\fe^2}}$ hold for every $0<\mathfrak e<1$.
\end{itemize}
\end{lem}
\begin{proof}
If $\mathfrak e=0$, we directly obtain $\mathcal{T}_{0}=2\pi$. We compute
$$
\mathcal{T}_{\mathfrak{e}}=\int_{0}^{2\pi}(1+\mathfrak{e}\cos\theta)^{-\frac{1}{2}}d\theta
=\int_{0}^{2\pi}\frac{1}{\sqrt{1+\mathfrak{e}(1-2\sin^2\frac{\theta}{2})}}d\theta
=\frac{4}{\sqrt{1+\mathfrak{e}}}K(\sqrt{\frac{2\mathfrak{e}}{1+\mathfrak{e}}}).\nonumber
$$
where $K(m):=\int_{0}^{\pi/2}(1-m^2\sin^2\theta)^{-\frac{1}{2}}d\theta,m\in[0,1)$, is the complete elliptic integral of the first kind, which satisfies $\lim_{m\rightarrow 1^+}K(m)=+\infty$. Then we obtain $\lim_{\mathfrak{e}\rightarrow 1^+}\mathcal{T}_{\mathfrak{e}}=+\infty$. Hence, (i) holds.

The derivative of $\mathcal T_{\mathfrak e}$ in $\mathfrak e$ can be computed as
$$
\frac{\partial\mathcal{T}_{\mathfrak{e}}}{\partial\mathfrak{e}}=-\int_{0}^{2\pi}\frac{\cos\theta}{2(1+\mathfrak{e}\cos\theta)^{\frac{3}{2}}}d\theta
=-\int_{-\frac{\pi}{2}}^{\frac{\pi}{2}}\frac{\cos\theta}{2(1+\mathfrak{e}\cos\theta)^{\frac{3}{2}}}d\theta
+\int_{-\frac{\pi}{2}}^{\frac{\pi}{2}}\frac{\cos\theta}{2(1-\mathfrak{e}\cos\theta)^{\frac{3}{2}}}d\theta.
$$
We see that $\frac{\partial\mathcal{T}_{\mathfrak{e}}}{\partial\mathfrak{e}}\big{|}_{\mathfrak{e}=0}=0$ and $\frac{\partial\mathcal{T}_{\mathfrak{e}}}{\partial\mathfrak{e}}>0$ for every $\mathfrak{e}\in(0,1)$. Moreover, since $\sqrt{1+\fe\cos\theta}\leq 1+\frac{1}{2}\fe\cos\theta$, we obtain $\mathcal{T}_{\mathfrak{e}}\geq\int_{0}^{2\pi}(1+\frac{1}{2}\mathfrak{e}\cos\theta)^{-1}d\theta=\frac{4\pi}{\sqrt{4-\fe^2}}$.
This completes the proof.
\end{proof}
Recall $\gamma_{\beta,\mathfrak e}$ the fundamental solution of equation \eqref{equ: linearized equation rp} and
$\mathcal{T}_{\mathfrak{e}}=\int_{0}^{2\pi}(1+\mathfrak{e}\cos\theta)^{-\frac{1}{2}}d\theta=\tau(2\pi)$. Let $S(\tau):=\mathrm{diag}(q^{\frac{1}{2}}(\tau),q^{-\frac{1}{2}}(\tau))\in \Sp(2)$ and $\hat{\gamma}_{\beta,\mathfrak{e}}(\tau):=S(\tau)\gamma_{\beta,\mathfrak{e}}(\tau)S^{-1}(0)$. Then $S(\mathcal{T}_\mathfrak{e})=S(0)$ and we compute that $\hat{\ga}_{\beta,\mathfrak{e}}$ satisfies
\begin{equation}\label{equ: blow up equation}
\frac{d\hat{\gamma}_{\beta,\mathfrak{e}}(\tau)}{d\tau}=J_{2}\hat{B}(\tau)\hat\gamma_{\beta,\mathfrak{e}}(\tau),\quad \hat{\gamma}_{\beta,\mathfrak{e}}(0)=I_2, \quad \text{with}\quad
\hat{B}(\tau)=\begin{pmatrix}
  1 & \frac{1}{4}Q \\
  \frac{1}{4}Q & q^{2}+\beta
\end{pmatrix}.
\end{equation}
Since the symplectic loop $S|_{\tau\in[0,\mathcal T_{\mathfrak e}]}$ is contractible in $\Sp(2)$, the homotopy invariance of the mean index shows that for every $0\leq \mathfrak{e}<1$, $1+\beta>0$ and $\omega\in \mathbf{U}$, we have $\hat{\gamma}_{\beta,\mathfrak{e}}(\mathcal{T}_{\mathfrak{e}}) \approx\gamma_{\beta,\mathfrak{e}}(2\pi)$ and
\bea\label{equ: mean index of hat gamma}
i_{\omega}(\hat{\gamma}_{\beta,\mathfrak{e}})=i_{\omega}(\gamma_{\beta,\mathfrak{e}}),\quad \hat{i}(\hat{\gamma}_{\beta,\mathfrak{e}})=\hat{i}(\gamma_{\beta,\mathfrak{e}}),
\quad \nu_{\omega}(\hat{\gamma}_{\beta,\mathfrak{e}})=\nu_{\omega}(\gamma_{\beta,\mathfrak{e}}).
\eea

Now we are ready to prove Theorem \ref{thm: comparison mean index 4}.
\begin{proof}[Proof of Theorem \ref{thm: comparison mean index 4}]
Let $\hat{\gamma}_{\beta,\mathfrak{e}}(\tau):=\big(a_{ij}(\tau)\big),i,j=1,2$ and $\Lambda_{D}:=\mathbb{R}\oplus\{0\}$. By \eqref{mean index2}, we have
\bea\label{mean-index-esti1}
\hat{i}(\hat{\gamma}_{\beta,\mathfrak{e}})=\lim_{k\to +\infty}\frac{1}{k}\sum_{0<\tau\leq k\mathcal T_{\mathfrak e}} \dim \hat{\gamma}_{\beta,\mathfrak{e}}\left(\tau\right) \Lambda_{D} \cap \Lambda_{D}
=\lim_{k\to +\infty}\frac{1}{k}\cdot\sharp\{\tau \in [0, k\mathcal T_{\mathfrak e}]: a_{21}(\tau)=0\}.
\eea
By equations \eqref{equ: blow up equation} and \eqref{equ: q,Q equation}, we compute that $a_{21}$ satisfies
\begin{equation}\label{equ: ddot a21}
\frac{d^2a_{21}}{d\tau^2}=\left[\frac{3}{16}Q^{2}(\tau)-\frac{3}{4}q^{2}(\tau)-\frac{1}{4}-\beta\right]a_{21}(\tau),\quad
\frac{da_{21}}{d\tau}\bigg|_{\tau=0}=1,\quad
a_{21}(0)=0.
\end{equation}
Using $q^2=1+\mathfrak{e}\cos\theta$ and $2E=q^2(Q^{2}+q^2-2)=\mathfrak{e}^2-1$, see \eqref{equ: first integral E}, we obtain
$$
f(q^2):=\frac{3}{16}Q^{2}-\frac{3}{4}q^{2}-\frac{1}{4}-\beta=\frac{3(\mathfrak{e}^2-1)}{16q^2}-\frac{15}{16}q^2+\frac{1}{8}-\beta,\quad\forall q^2\in[1-\mathfrak{e},1+\mathfrak{e}]\subset [0,+\infty),
$$
which increases on $(0,(\frac{1-\mathfrak{e}^2}{5})^{\frac{1}{2}})$ and decreases on $((\frac{1-\mathfrak{e}^2}{5})^{\frac{1}{2}},+\infty)$. To establish the inequality $\hat \rho_{p,\beta,\fe}>\hat\rho_{p,\beta,0}$ for every $0<\mathfrak e<1$ and $\beta>\beta_*(\mathfrak e)$, we would split in two cases: $2/3\leq\mathfrak{e}<1$ and $0<\mathfrak{e}<2/3$.

Assume $2/3\leq\mathfrak{e}<1$. Then $(\frac{1-\mathfrak{e}^2}{5})^{\frac{1}{2}}\in[1-\fe,1+\fe]$ and $f(q^2)\leq f((\frac{1-\mathfrak{e}^2}{5})^{\frac{1}{2}}) = \frac{1}{8}(-3\sqrt{5}\sqrt{1-\mathfrak{e}^2}+1)-\beta$. For every $\beta\geq\frac{1}{8}(-3\sqrt{5}\sqrt{1-\mathfrak{e}^2}+1)$, the comparison theorem for \eqref{equ: ddot a21} implies that
\bea\label{esim4}
\begin{aligned}
\sharp\left\{\tau \in [0, \tau_{0}]: a_{21}(\tau)=0\right\}&\geq\sharp\{\tau \in [0, \tau_{0}]: \tilde a_{21}(\tau)=0\} \\ &=\left\lfloor\frac{\tau_{0}}{\pi}\left(\frac{1}{8}\left(3\sqrt{5}\sqrt{1-\mathfrak{e}^2}-1\right)+\beta\right)^{\frac{1}{2}}\right\rfloor+1,
\end{aligned}\eea
where $\tilde a_{21}$ solves equation
$$
\frac{d^2\tilde a_{21}}{d\tau^2}=-\left(\frac{1}{8}\left(3\sqrt{5}\sqrt{1-\mathfrak{e}^2}-1\right)+\beta\right)a_{21}(\tau),\quad
\frac{d\tilde a_{21}}{d\tau}\bigg|_{\tau=0}=1,\quad \tilde a_{21}(0)=0.
$$
Then using \eqref{equ: mean index of hat gamma}, \eqref{mean-index-esti1} and \eqref{esim4}, we conclude that
$$
\hat{i}(\gamma_{\beta,\mathfrak{e}})=\hat{i}(\hat{\gamma}_{\beta,\mathfrak{e}})
\geq\frac{\mathcal{T}_{\mathfrak{e}}}{\pi}\left(\frac{1}{8}\left(3\sqrt{5}\sqrt{1-\mathfrak{e}^2}-1\right)+\beta\right)^{\frac{1}{2}}
>2\sqrt{1+\beta},
$$
for every
$$\beta>\beta_*(\fe):=\frac{1-3\sqrt{5}\sqrt{1-\mathfrak{e}^2}}{8} +\frac{3\pi^2(3-\sqrt{5}\sqrt{1-\mathfrak{e}^2})}{2(\mathcal{T}^2_{\mathfrak{e}}-4\pi^2)}.$$
Hence, we obtain $\hat\rho_{p,\beta,\mathfrak{e}}>\hat \rho_{p,\beta,0}=\sqrt{1+\beta}$ for every $\beta>\beta_*(\mathfrak e)$. Moreover, by Lemma \ref{lem: Perod}-(ii), we obtain $\mathcal{T}_{\mathfrak{e}}^2-4\pi^2\geq4\pi^2\fe^2/(4-\fe^2)$ and
$$\begin{aligned}
\beta_*(\fe)&\leq\frac{1}{8}(1-3\sqrt{5}\sqrt{1-\mathfrak{e}^2})+\frac{3}{2}(3-\sqrt{5}\sqrt{1-\mathfrak{e}^2})(\frac{1}{\fe^2}-\frac{1}{4})\\
&=-1+\frac{3}{2\fe^2}(3-\sqrt{5}\sqrt{1-\fe^2})\leq7/2<6/\fe+3,\quad \forall \mathfrak e\in[2/3,1).
\end{aligned}$$
Therefore, Theorem \ref{thm: comparison mean index 4} is proved for every $2/3\leq \mathfrak e<1$.

Assume $0\leq \fe<2/3$. Then $1-\fe>(\frac{1-\mathfrak{e}^2}{5})^{\frac{1}{2}}$ and $f(q^2)\leq f(1-\mathfrak e)= 3\mathfrak e/4-1-\beta<0$. Therefore, for every $\beta\geq 3\mathfrak{e}/4-1$, the relation \eqref{equ: mean index of hat gamma} and the comparison theorem for \eqref{equ: ddot a21} imply that
$$
\hat{i}(\gamma_{\beta,\mathfrak{e}})=\hat{i}(\hat{\gamma}_{\beta,\mathfrak{e}})
\geq\frac{\mathcal{T}_{\mathfrak{e}}}{\pi}\left(-\frac{3\mathfrak{e}}{4}+1+\beta\right)^{\frac{1}{2}}\geq 2\sqrt{1+\beta},
$$
for every
$$\beta>\beta_*(\fe):=\frac{3\mathfrak{e}-4}{4}+\frac{3\pi^2\mathfrak{e}}{\mathcal{T}^2_{\mathfrak{e}}-4\pi^2}.$$
Hence, we have $\hat \rho_{p,\beta,\mathfrak{e}}\geq \hat \rho_{p,\beta,0}=\sqrt{1+\beta}$ for every $\beta>\beta_*(\mathfrak e)$.
Moreover, since $\mathcal{T}_{\mathfrak{e}}^2-4\pi^2\geq4\pi^2\fe^2/(4-\fe^2)$, we have
$\beta_*(\fe)\leq 3/\fe-1<6/\fe+3.$ This proves Theorem \ref{thm: comparison mean index 4} for every $0 <\mathfrak e<2/3$. Hence, the proof is now complete.
\end{proof}
\subsection{Proof of Theorem \ref{thm: estimate of rotation number}}

Recall that $\beta=D/C-1$, $D=\sum_{i=1}^{\mathfrak{n}}A_iB_i>0$ and $C=\sum_{i=0}^{\mathfrak{n}}A_i>0$. The condition $D\geq C$ implies $\beta\geq0$. When $A_0=0$ and $B_i=1,\,i=1,\cdots,\mathfrak{n}$, the system reduces to the reduced spatial Kepler problem and the result has been proved in \cite{HOQ26}. Then we assume $A_0\neq0$ or $B_k\neq1$ for some $k\in\{1,\cdots,\mathfrak{n}\}$. If $\beta=0$, by Theorem \ref{thm: prop of deg curve}, we obtain $\hat \rho_p=1$. When $A_0=0$, we can derive that there exist $i\neq j$, such that $B_i\neq B_j$, by Proposition \ref{prop, vol1 vol2}, we know $\mathrm{vol} (\mathfrak{M},\lambda)>\mathrm{vol}(\hat{\mathfrak{M}}_0,\hat{\lambda}_0)$. When $A_0>0$, by Proposition \ref{prop, vol1 vol2}, we obtain $\mathrm{vol} (\mathfrak{M},\lambda)>\mathrm{vol}(\hat{\mathfrak{M}}_0,\hat{\lambda}_0)$ again. Combined with Theorem \ref{thm: vol of M}, we obtain $$\hat \rho_{p}=1>\mathrm{vol}(\hat{\mathfrak{M}}_0,\hat{\lambda}_0)/\mathrm{vol} (\mathfrak{M},\lambda)=T_p^2/\mathrm{vol}(\mathfrak{M},\lambda).$$

Then we assume $\beta>0$. Combining Theorems \ref{thm: comparison mean index 3} and \ref{thm: comparison mean index 4}, we obtain the inequality $\hat \rho_{p,\beta,\mathfrak e}>\hat \rho_{p,\beta,0}=\sqrt{1+\beta}$ for every $\beta\in (0,+\infty)$ and $0<\mathfrak e<1$. By Proposition \ref{prop, vol1 vol2} and Theorem \ref{thm: vol of M}, we have $\mathrm{vol}(\mathfrak{M},\lambda)\geq\mathrm{vol}(\hat{\mathfrak{M}}_0,\hat{\lambda}_0)=T_p^2/\sqrt{1+\beta}$. So we obtain
$$\hat \rho_{p}>\sqrt{1+\beta}\geq T_p^2/\mathrm{vol}(\mathfrak{M},\lambda).
$$
By Corollary \ref{cor: infini peri orbit1}, the energy surface $\mathfrak M$ admits infinitely many periodic orbits, the proof of Theorem \ref{thm: estimate of rotation number} is complete.

\subsection{Proof of Corollary  \ref{coro: estimate of rotation number 1}}
Let $\mathfrak{n}=n+1\geq3$, $A_0=a(n)$, $A_i=b_{n,i}/4,\,i=1,\cdots,n$, $A_{\mathfrak{n}}=m_0$ and $B_i=b_{n,i}^2,\,i=1,\cdots,n$, $B_{\mathfrak{n}}=1$. The Hamiltonian \eqref{equ: reduced Ham} becomes Hamiltonian \eqref{equ: hiphop ham}. Let $c(n,m_0)=a(n)+b(n)+m_0$ and $d(n,m_0)=\frac{1}{4}\sum_{k=1}^{n}b_{n,k}^3+m_0$. So we only need to verify $d(n,m_0)>c(n,m_0)$.
Direct computation shows that
$$\begin{aligned}
	\frac{d(n,0)}{c(n,0)}=&\frac{\sum_{k=1}^{n}\csc^{3} ((2k-1)\pi/2n)}{\sum_{k=1}^{2n-1}\csc (k\pi/2n)}>\frac{2\csc^{3}(\pi/2n)}{1+(2n-2)\csc(\pi/2n)}\\
	=&2(\sin^{3}(\pi/2n)+(2n-2)\sin^{2}(\pi/2n))^{-1}>8n^{3}(\pi^{3}/2+2n(n-1)\pi^{2})^{-1}.
\end{aligned}$$
Hence $\lim_{n\rightarrow+\infty}d(n,0)/c(n,0)=+\infty.$ Moreover, let $f(x)=8x^{3}(\pi^{3}/2+2x(x-1)\pi^{2})^{-1}$, $f(x)$ is strictly increasing for $x\geq2$. Since $f(2)=64(\pi^{3}/2+4\pi^{2})^{-1}>1,$ we have $d(n,0)>c(n,0)$ for all $n\geq2$. So $d(n,m_0)>c(n,m_0)$ for all $m_0\geq0$.
\begin{rem}
	For fixed $n$, $\beta(n,m_0)=d(n,m_0)/c(n,m_0)-1$ is strictly decreasing with respect to $m_0$. So we conclude that:
	\begin{equation*}
		0<=\beta(n,m_0)\leq\frac{d(n,0)}{c(n,0)}-1,
	\end{equation*}
	and
	$$\lim_{n\rightarrow+\infty}\beta(n,m_{0})=+\infty,\ \ \lim_{m_{0}\rightarrow+\infty}\beta(n,m_{0})=0.$$
	For a compact energy surface, we have $0<\mathfrak{e}<(1-a(n)^{2}/c(n,m_0)^{2})^{\frac{1}{2}}$. Since $\beta(n,m_0)=d(n,m_0)/c(n,m_0)-1$, so it can be expressed as $0<\mathfrak{e}<(1-a(n)^{2}\beta(n,m_0)^{2}/(d(n,0)-c(n,0))^{2})^{\frac{1}{2}}$. For a fixed $n$, when we consider the $(\beta,\mathfrak{e})-$parameter curve, the range corresponding to a compact energy surface is exactly a portion of a quarter-ellipse, see Figure \ref{fig:Upsilon 1}.
\end{rem}
\begin{figure}[hpbt]
\centering
\includegraphics[width=0.55\textwidth]{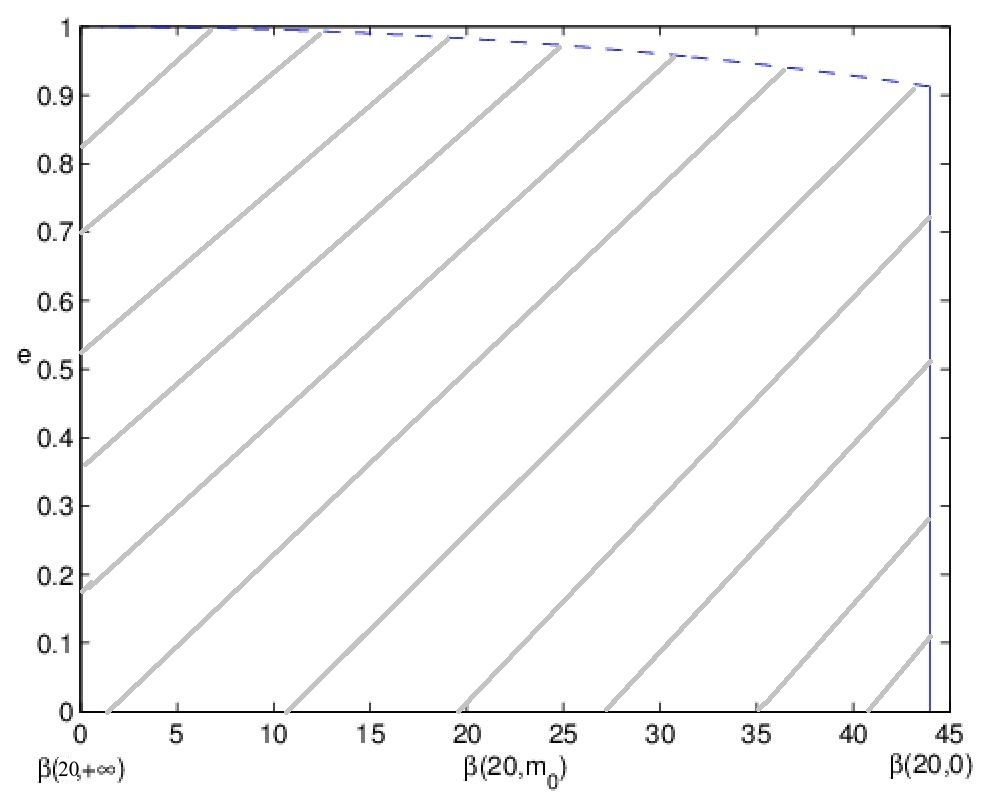}
\caption{When
n=20
, the parameter region that ensures the energy surface is compact.}
\label{fig:Upsilon 1}
\end{figure}

\subsection{Proof of Corollary \ref{coro: estimate of rotation number 2}}

Let $\mathfrak{n}=1$, $A_0=a(n)\alpha$, $A_1=1$ and $B_1=1+n\alpha$, The Hamiltonian \eqref{equ: reduced Ham} coincide with Hamiltonian \eqref{equ: reduced Ham 2}. Since $2\leq n\leq472$ is equalent to $n>a(n)$, which satisfies the condition of Theorem \ref{thm: estimate of rotation number}.

\appendix
\section{The Maslov-type index theory}\label{sec: App}
In this section, we briefly introduce the Maslov-type index, see \cite{Lon02} for more details. Let $\Sp(2n)$ denote the set of $2n\times2n$ real symplectic matrix. Denote $\mathcal{P}_{\tau}(2n):=\left\{\gamma \in C([0,\tau], \Sp(2n)): \gamma(0)=I_{2 n}\right\}$. For every $\omega\in\mathbf{U}$, we define the $\omega$-degenerate hypersurface in $\Sp(2n)$ as
$$
\Sp(2n)_\omega^0:=\left\{M\in \Sp(2n): \det(M-\omega I_{2n})=0 \right\},
$$
on which there exists a nowhere vanishing vector field $V(M):=\frac{d}{dt}Me^{tJ_{2n}}|_{t=0},\forall M\in \Sp(2n)_\omega^0$ that is everywhere transverse to $\Sp(2n)_\omega^0$ and determines a positive co-orientation of $\Sp(2n)_\omega^0$. Denote the $\omega$-regular set in $\Sp(2n)$ as $\Sp(2n)^*_{\omega}:=\Sp(2n)\setminus \Sp(2n)_\omega^0$. We define the $\omega$-index, $\omega$-nullity and the mean index of a symplectic path as follows.
\begin{defi}\label{def:Maslov-type index}
Let $\gamma\in \mathcal{P}_{\tau}$. For every $\omega\in \mathbf{U}$, we define the $\omega$-index and $\omega$-nullity of $\gamma$ as
$$
i_\omega(\gamma):=\begin{cases}(e^{-\epsilon J_{2n}}\gamma)\cdot \Sp(2n)_\omega^0-n, & \text { if } \omega=1 \\
(e^{-\epsilon J_{2n}}\gamma)\cdot \Sp(2n)_\omega^0, & \text { if } \omega\neq1.\end{cases} \quad
\nu_\omega(\gamma):=\dim_{\mathbb{C}}\ker_{\mathbb{C}}(\gamma(\tau)-\omega I_{2n}),
$$
where $(e^{-\epsilon J_{2n}}\gamma)\cdot \Sp(2n)_\omega^0$ denote the algebraic intersection number and $\epsilon>0$ is sufficiently small. Define the mean index of $\gamma$ as $\hat{i}(\gamma):=\lim_{m\rightarrow+\infty}i_{1}(\gamma^m)/m$, where $\gamma^m:[0,k\tau]\rightarrow \Sp(2n)$ denote the $m$-th iteration of $\gamma$ for every $m\in  \mathbb Z_+$, which is defined as
$$
\gamma^m(t):=\gamma(t-j\tau)\gamma(\tau)^j,\quad\forall t\in[j\tau,(j+1)\tau],\quad j=0,\cdots,m-1.
$$
\end{defi}
Note that if $\omega=1$, the $1$-index coincides with the Conley-Zehnder index. In $\Sp(2)$, there is another way of understanding to the $1$-index. Let $\eta(t)\in C([0,\tau],\mathbb R)$ be the argument of $\gamma(t)v$ for any $v\in \mathbf{U}$. Let $\Delta(v):=\frac{1}{2\pi}(\eta(\tau) - \eta(0))$ be the variation of $\eta$ on $[0,\tau]$. Define $I_\gamma:= \{\Delta(v): v \in \mathbf{U}\} \subset \mathbb{R}$, which is a closed interval of length $<1/2$. For every $\epsilon>0$ sufficiently small, we define $i_1(\gamma):=2k+1$ if $I_\gamma-\epsilon\in (k,k+1)$ and define $i_1(\gamma): = 2k$ if $k\in I_\gamma-\epsilon$. The rotation number of $\gamma$ is defined as
\begin{equation}\label{def: mean index and rotation number}
\rho(\gamma) :=\frac{\hat i(\gamma)}{2}= \lim_{k\to \infty}\frac{\eta(k\tau)}{2\pi k}.
\end{equation}

Let $\Lambda_D:=\mathbb{R}\oplus\{0\}$ be a Lagrangian subspace of $\rr^2$. Assume the path $\gamma\in \mathcal P_\tau(2n)$ is differentiable and satisfies $-J_{2n}\dot \gamma(t)\gamma(t)^{-1}|_{\Lambda_D}>0$ for every $t\in \mathbb R$. As the Morse index theorem,
the mean index can be computed as
\begin{eqnarray}\label{mean index2}
\hat{i}(\gamma)=\lim_{k\to +\infty}\frac{1}{k}\sum_{0<t\leq k\tau} \operatorname{dim} \gamma(t) \Lambda_{D} \cap \Lambda_{D}.
\end{eqnarray}


The $\omega$-index also relates to the Morse index of a certain self-adjoint operator. Consider a second order system $\ddot{x}=\mathcal D(t)x$, $t\in[0,\tau]$. Let $\mathcal{A}:=-\frac{d^2}{dt^2}+\mathcal D$ be a self-adjoint operator on $L^2([0,\tau],\mathbb{C}^n)$ with domain
$$
D(\omega,\tau):=\left\{W^{2,2}([0,\tau], \mathbb{C}^n): y(\tau)=\omega y(0), \dot{y}(\tau)=\omega \dot{y}(0)\right\}.
$$
Let $m^-_\omega(\mathcal{A})$ denote the Morse index of $\mathcal A$, which is the total multiplicity of the negative eigenvalues of $\mathcal A$. Let $\nu_{\omega}(A):=\dim\ker(A)$ denote the nullity of $\mathcal{A}$. From Theorem 7.3.4 in \cite{Lon02}, we have the following relations
\begin{equation}\label{index equ}
m^-_{\omega}(\mathcal{A})=i_{\omega}(\gamma), \quad \nu_{\omega}(\mathcal{A})=\nu_{\omega}(\gamma),\quad \forall \omega\in \mathbf U,
\end{equation}
where $\gamma\in \mathcal P_\tau(2n)$ is the fundamental solution of the linear system
$\dot{\gamma}=J_{2n} \mathrm{diag}(I_n,-D(t))\gamma$. For more general boundary conditions, we refer to \cite{HWY} for a similar result.

\hfill\newline
\noindent{\bf Acknowledgement.}
X.Hu, Z.Qiao and Y.Yang are partially supported by the National Natural Science Foundation of China (No.12521001) and Taishan Scholars Climbing Program of Shandong(No.TSPD20240802). Y.Ou is partially supported by the National Natural Science Foundation of China (No.12371192), the Young Taishan Scholars Program of Shandong Province (No.tsqn202312055), and the Qilu Young Scholar Program of Shandong University.
\bibliographystyle{abbrv}
\bibliography{references}

@inproceedings{Ale72,
    author  = {Alekseev, V. M.},
    title   = {Quasirandom oscillations and qualitative problems in celestial mechanics},
    booktitle = {Ninth Mathematical Summer School (Kaciveli, 1971) (Russian)},
    pages   = {212--341},
    year    = {1972},
    note    = {Three papers on smooth dynamical systems}
}

@article{BTG14,
    author  = {Barutello, V. and Terracini, S. and Verzini, G.},
    title   = {Entire parabolic trajectories as minimal phase transitions},
    journal = {Calc. Var. Part. Diff. Equa.},
    volume  = {49},
    number  = {1-2},
    pages   = {391--429},
    year    = {2014}
}

@article{CHHL23,
    author  = {Cristofaro-Gardiner, D. and Hryniewicz, U. and Hutchings, M. and Liu, H.},
    title   = {Contact three-manifolds with exactly two simple {R}eeb orbits},
    journal = {Geom. Topo.},
    volume  = {27},
    number  = {9},
    pages   = {3801--3831},
    year    = {2023}
}

@article{CHHL26,
    author  = {Cristofaro-Gardiner, D. and Hryniewicz, U. and Hutchings, M. and Liu, H.},
    title   = {Proof of {H}ofer-{W}ysocki-{Z}ehnder's two or infinity conjecture},
    journal = {J. Am. Math. Soc.},
    year    = {2026},
    note    = {online first}
}

@article{GMC73,
    author  = {Gutzwiller, M. C.},
    title   = {The anisotropic {K}epler problem in two dimensions},
    journal = {J. Math. Phys.},
    volume  = {14},
    pages   = {139--152},
    year    = {1973}
}

@book{GMC1990,
    author  = {Gutzwiller, M. C.},
    title   = {Chaos in Classical and Quantum Mechanics},
    volume  = {1},
    series  = {Interdisciplinary Applied Mathematics},
    publisher = {Springer-Verlag},
    address = {New York},
    year    = {1990}
}

@article{HLOQS26,
    author  = {Hu, X. and Liu, L. and Ou, Y. and Qiao, Z. and Salom\~{a}o, P. A. S.},
    title   = {{ECH} constraints and twist dynamics in the spatial isosceles three-body problem},
    journal = {Geom. Topol.},
    year    = {2026},
note={arXiv:2602.24025, to appear in G\&T}
}

@article{HLOYS23,
    author  = {Hu, X. and Liu, L. and Ou, Y. and Salom\~{a}o, P. A. S. and Yu, G.},
    title   = {A symplectic dynamics approach to the spatial isosceles three-body problem},
    journal = {J. Euro. Math. Soci. (2025), online first},
    year    = {2025}
}

@article{HO16,
    author  = {Hu, X. and Ou, Y.},
    title   = {Collision index and stability of elliptic relative equilibria in planar \(n\)-body problem},
    journal = {Comm. Math. Phys.},
    volume  = {348},
    number  = {3},
    pages   = {803--845},
    year    = {2016}
}

@article{HOQ26,
    author  = {Hu, X. and Ou, Y. and Qiao, Z.},
    title   = {Relative periodic orbits in the spatial anisotropic {K}epler problem},
    journal = {Discrete Contin. Dyn. Syst.},
    year    = {2026},
    volume  = {52},
pages   = {434--447}
}

@article{HOT23,
    author  = {Hu, X. and Ou, Y. and Tang, X.},
    title   = {Linear stability of an elliptic relative equilibrium in the spatial \(n\)-body problem via index theory},
    journal = {Regu. Chao. Dyna.},
    volume  = {28},
    pages   = {731--755},
    year    = {2023}
}

@article{HQY25,
    author  = {Hu, X. and Qiao, Z. and Yu, G.},
    title   = {Relative periodic solutions in spatial {K}epler problem with symmetric perturbation},
    journal = {Nonlinearity},
    volume  = {39},
    year    = {2026},
    pages   = {025001}
}

@article{HWY,
    author  = {Hu, X. and Wu, L. and Yang, R.},
    title   = {Morse index theorem of {L}agrangian systems and stability of brake orbit},
    journal = {J. Dyna. Diff. Equa.},
    volume  = {32},
    number  = {1},
    pages   = {61--84},
    year    = {2020}
}

@article{HY18,
    author  = {Hu, X. and Yu, G.},
    title   = {An index theory for zero energy solutions of the planar anisotropic {K}epler problem},
    journal = {Comm. Math. Phys.},
    volume  = {361},
    number  = {2},
    pages   = {709--736},
    year    = {2018}
}

@book{Lon02,
    author  = {Long, Y.},
    title   = {Index Theory for Symplectic Paths with Applications},
    volume  = {207},
    series  = {Progress in Mathematics},
    publisher = {Birkh\"{a}user Verlag},
    address = {Basel},
    year    = {2002}
}

@book{MW1966,
    author  = {Magnus, W. and Winkler, S.},
    title   = {Hill's Equation},
    publisher = {Interscience Publishers, Wiley},
    address = {New York},
    year    = {1966}
}

@article{M84,
    author  = {Moeckel, R.},
    title   = {Heteroclinic phenomena in the isosceles three-body problem},
    journal = {SIAM J. Math. Anal.},
    volume  = {15},
    number  = {5},
    pages   = {857--876},
    year    = {1984}
}

@article{Sit60,
    author  = {Sitnikov, K.},
    title   = {The existence of oscillatory motions in the three-body problems},
    journal = {Sovi. Phys. Dokl.},
    volume  = {5},
    pages   = {647--650},
    year    = {1960}
}

@article{Y25,
    author  = {Yu, G.},
    title   = {Positive energy solutions in the anisotropic {K}epler problem with homogeneous potential},
    journal = {Commun. Math. Phys.},
volume={407},
    year    = {2026}
}

@article{davies1983classical,
    author  = {Davies, I. and Truman, A. and Williams, D.},
    title   = {Classical periodic solutions of the equal-mass \(2n\)-body problem, \(2n\)-ion problem and the \(n\)-electron atom problem},
    journal = {Phys. Lett. A},
    volume  = {99},
    number  = {1},
    pages   = {15--18},
    year    = {1983}
}

@article{Chenciner2000,
    author  = {Chenciner, A. and Venturelli, A.},
    title   = {Minima de l'int\'{e}grale d'action du probl\`{e}me newtonien de 4 corps de masses \'{e}gales dans {{\(\mathbf{R}^3\)}}: orbites ``hip-hop''},
    journal = {Cele. Mech. Dyna. Astr.},
    volume  = {77},
    number  = {2},
    pages   = {139--152},
    year    = {2000}
}

@article{Terracini2007,
    author  = {Terracini, S. and Venturelli, A.},
    title   = {Symmetric trajectories for the $2N$-body problem with equal masses},
    journal = {Arch. Rati. Mech. Anal.},
    volume  = {184},
    number  = {3},
    pages   = {465--493},
    year    = {2007}
}

@article{Barrabes2006,
    author  = {Barrab\'{e}s, E. and Cors, J. M. and Pinyol, C. and Soler, J.},
    title   = {Hip-hop solutions of the \(2N\)-body problem},
    journal = {Cele. Mech. Dyna. Astr.},
    volume  = {95},
    number  = {1--4},
    pages   = {55--66},
    year    = {2006}
}

@article{Barrabes2010,
    author  = {Barrab\'{e}s, E. and Cors, J. M. and Pinyol, C. and Soler, J.},
    title   = {Highly eccentric hip-hop solutions of the \(2N\)-body problem},
    journal = {Physica D},
    volume  = {239},
    number  = {5},
    pages   = {214--219},
    year    = {2010}
}

@article{MSS06,
	author  = {Mart\'{i}nez, R. and Sam\`{a}, C. and Sim\'{o}, C.},
	title   = {Analysis of the stability of a family of singular-limit linear periodic systems in {{\(\mathbb{R}^4\)}}. {A}pplications},
	journal = {J. Diff. Equa.},
	volume  = {226},
	number  = {2},
	pages   = {652--686},
	year    = {2006}
}

@article {GLV13,
    AUTHOR = {Guirao, Juan L. G. and Llibre, Jaume and Vera, Juan A.},
     TITLE = {Periodic orbits of {H}amiltonian systems: applications to
              perturbed {K}epler problems},
   JOURNAL = {Chaos Solitons Fractals},
  FJOURNAL = {Chaos, Solitons \& Fractals},
    VOLUME = {57},
      YEAR = {2013},
     PAGES = {105--111},
      ISSN = {0960-0779,1873-2887},
   MRCLASS = {70H12 (70F05 70F15)},
  MRNUMBER = {3128604},
MRREVIEWER = {Shanzhong\ Sun},
       DOI = {10.1016/j.chaos.2013.09.003},
       URL = {https://doi.org/10.1016/j.chaos.2013.09.003},
}

@article {LM12,
    AUTHOR = {Llibre, Jaume and Makhlouf, Ammar},
     TITLE = {Periodic orbits of the spatial anisotropic {M}anev problem},
   JOURNAL = {J. Math. Phys.},
  FJOURNAL = {Journal of Mathematical Physics},
    VOLUME = {53},
      YEAR = {2012},
    NUMBER = {12},
     PAGES = {122903, 7},
      ISSN = {0022-2488,1089-7658},
   MRCLASS = {70F05 (70K42)},
  MRNUMBER = {3058197},
MRREVIEWER = {Pura\ Vindel},
       DOI = {10.1063/1.4771902},
       URL = {https://doi.org/10.1063/1.4771902},
}

@book {HZ94,
    AUTHOR = {Hofer, Helmut and Zehnder, Eduard},
     TITLE = {Symplectic invariants and {H}amiltonian dynamics},
    SERIES = {Birkh\"auser Advanced Texts: Basler Lehrb\"ucher.
              [Birkh\"auser Advanced Texts: Basel Textbooks]},
 PUBLISHER = {Birkh\"auser Verlag, Basel},
      YEAR = {1994},
     PAGES = {xiv+341},
      ISBN = {3-7643-5066-0},
   MRCLASS = {58-02 (34C25 57R15 58E05 58F05 70H05)},
  MRNUMBER = {1306732},
MRREVIEWER = {Daniel\ M.\ Burns, Jr.},
       DOI = {10.1007/978-3-0348-8540-9},
       URL = {https://doi.org/10.1007/978-3-0348-8540-9},
}

@article{SakaShib26,
    author  = {Shu Sakaguchi and Mitsuru Shibayama},
    title   = {A Minimax Approach to Relative Periodic Orbits in Symmetric Three-Degree-of-Freedom {H}amiltonian Systems},
    year    = {2026},
note={arXiv:2607.00517},
}

	\end{document}